\documentclass[11pt]{amsart}

\usepackage{amsmath, amssymb}
\usepackage{hyperref}
\usepackage{enumitem}
\newcounter{my_enumerate_counter}
\newcommand{\pushcounter}{\setcounter{my_enumerate_counter}{\value{enumi}}}
\newcommand{\popcounter}{\setcounter{enumi}{\value{my_enumerate_counter}}}

\DeclareMathOperator{\bbM}{Ult}

\DeclareMathOperator{\Fin}{Fin}

\DeclareMathOperator{\ZFC}{ZFC}

\DeclareMathOperator{\invp}{inv}
\DeclareMathOperator{\MODc}{Mod_{\fc}}

\newcommand{\reg}{\lessdot}

\newcommand{\bbNN}{\bbN^{\bbN}}

\newcommand{\bbH}{\mathbb H}
\newcommand{\bbA}{\mathbb A} 
\newcommand{\bbB}{\mathbb B} 
\newcommand{\bbN}{{\mathbb N}}
\newcommand{\bbC}{\mathbb C}

\newcommand{\bbR}{\mathbb R}
\newcommand{\cJ}{{\mathcal J}}
\newcommand{\cI}{{\mathcal I}}

\newcommand{\calL}{\mathcal L}
\newcommand{\cX}{{\mathcal X}}

\newcommand{\cA}{{\mathcal A}}
\newcommand{\cE}{{\mathcal E}}

\newcommand{\fc}{\mathfrak c}

\newcommand{\rs}{\restriction}

\newcommand{\cU}{\mathcal U} 

\newcommand{\cV}{\mathcal V} 
\DeclareMathOperator{\dom}{dom}

\newcommand{\cF}{\mathcal F}

\newcommand{\cB}{\mathcal B}

\newcommand{\fB}{\mathfrak B}

\newcommand{\calD}{\mathcal D}
\newcommand{\bbP}{\mathbb P}
\newcommand{\e}{\varepsilon}
\newtheorem{thm}{Theorem}[section]
\newtheorem*{thms}{Theorem}
\newtheorem{theorem}{Theorem}

\newtheorem{corollary}[theorem]{Corollary}

\newtheorem{claim}[thm]{Claim}

\newtheorem{quest}[thm]{Question}
\newtheorem{lemma}[thm]{Lemma}

\newtheorem{prop}[thm]{Proposition}

\theoremstyle{definition}
\newtheorem{remark}[thm]{Remark}

\newtheorem{definition}[thm]{Definition}

\DeclareMathOperator{\inv}{inv} 
\DeclareMathOperator{\INV}{INV} 
\DeclareMathOperator{\cf}{cf} 

\numberwithin{equation}{section}

\newcommand{\cP}{\mathcal P}

\newcommand{\forces}{\Vdash}
\newcommand{\nprecc}{\ntriangleleft} 

\newcommand{\nleqE}{\nleq_E}
\newcommand{\lE}{<_E}
\newcommand{\leqEp}{\leq_{E'}}
\newcommand{\geqEp}{\geq_{E'}}

\newcommand{\leqE}{\leq_E}

\newcommand{\precc}{\triangleleft} 
\newcommand{\precp}{\triangleleft_\varphi} 
\newcommand{\nprecp}{\ntriangleleft_\varphi}

\newcommand{\cstar}{\textrm{$C^*$}}

\title{Between reduced powers and ultrapowers, II.}

\author{Ilijas Farah}

\address{Department of Mathematics and Statistics\\
York University\\
4700 Keele Street\\
North York, Ontario\\ Canada, M3J
1P3 and Matemati\c vki Institut SANU, Kneza Mihaila 36, Belgrade 11001, Serbia}

\urladdr{https://ifarah.mathstats.yorku.ca}

\email{ifarah@yorku.ca}

\author{Saharon Shelah}

\address{The Hebrew University of Jerusalem\\
Einstein Institute of Mathematics\\
Edmond J. Safra Campus, Givat Ram\\
Jerusalem 919041, Israel and 
Department of Mathematics\\
Hill Center-Busch Campus\\
Rutgers, The State University of New Jersey\\
110 Frelinghuysen Road\\
Piscataway, NJ 08854-8019 USA}

\email{shelah@math.huji.ac.il}
\urladdr{http://shelah.logic.at/}

\thanks{First author's research is partially supported by NSERC. Second author's research  partially supported by Israel Science Foundation (ISF) grant no: 1838/19 and  by NSF grant no: DMS 1833363. No. 1202 on  Shelah's list of publications.}



\date{\today}

\begin{document}
\begin{abstract} We prove that, consistently with ZFC, no ultraproduct of countably infinite (or separable metric, non-compact) structures is isomorphic to a reduced product of countable (or separable metric) structures associated to the Fr\'echet filter. Since such structures are countably saturated, the Continuum Hypothesis implies that they are isomorphic when elementarily equivalent. 
\end{abstract}
\maketitle

The trivializing effect of the Continuum Hypothesis (CH) to the structure of the continuum has been known at least since the times of Sierpi\' nski and G\"odel (\cite{godel1947cantor}). 
The particular instance of this phenomenon that we are concerned with in the present paper is the existence of highly non-canonical isomorphisms between massive quotient structures of cardinality $\fc=2^{\aleph_0}$.  
The operation of taking a reduced product\footnote{Reduced products can be defined with respect to a filter or  with respect to the dual ideal, the two definitions resulting in the same object; see Remark~\ref{Rem.ideal.filter}.} $\prod_{\cF} A_n$ of a sequence $(A_n)$ of first-order structures often results in a countably saturated structure.\footnote{Saturation is a model-theoretic property that enables transfinite constructions of isomorphisms and automorphisms; see e.g., \cite[Chapter 5]{ChaKe}. It will not be used in this paper.} 
This is the case with the two most commonly used reduced products: ultraproducts associated with nonprincipal ultrafilters on $\bbN$ and reduced products associated with the Fr\'echet filter.  If each~$A_n$ has the cardinality of  at most~$\fc$ (in particular, if it is countable or separable\footnote{Our results apply both to classical, discrete, structures and metric structures, as in~\cite{BYBHU}.}), then so does $\prod_{\cF} A_n$, and the CH implies that the latter structure is saturated. 
 By a  classical theorem of Keisler, elementarily equivalent and saturated first-order structures of the same cardinality are isomorphic (see \cite[Theorem~5.1.13]{ChaKe}). 
Therefore CH implies that the isomorphism of such reduced products reduces (no pun intended) to elementary equivalence. 

In   \cite{farah2020between}, this observation was combined with computation of the theory of the structure ($K$ denotes  the Cantor space) 
\begin{equation}\label{Eq.CKA}
	C(K,A)=\{f\colon K\to A\mid \text{ $f$ is continuous}\}
\end{equation}
 for a separable (or countable discrete) structure $A$  to prove that CH implies ($\Fin$ denotes the ideal of finite subsets of $\bbN$)
\begin{equation} \label{Eq.1} 
\textstyle \prod_{\cU} C(K,A)\cong \prod_{\Fin} A
\end{equation} 
for any nonprincipal ultrafilter $\cU$ on $\bbN$ (\cite[Theorem~E]{farah2020between}).\footnote{In \cite{farah2020between}, $\prod_{\Fin}A$ was denoted $A^\infty$ and $\prod_{\cU} A$ was denoted $A^{\cU}$, following the notation favoured by operator-algebraists. In the present paper we adopt the notation favoured by logicians and apologize to any stray operator algebraists; see however Corollary~\ref{C.D}.}
  This result is the basis for \cite[Theorem~B]{farah2020between}, asserting that under CH there exists an ultrafilter $\cU$ on $\bbN$ such that  the quotient map from $\prod_{\Fin} A$ to $\prod_{\cU} A$ has a right inverse for every countable (or separable metric) structure $A$.  In the case when~$A$ is a \cstar-algebra, this  simplifies some intricate arguments in  Elliott's classification program for nuclear, separable \cstar-algebras (see  the upcoming~\cite{CGSTW}, also \cite{winter2018structure} and \cite{Ror:Classification}  for related applications of ultrapowers). 
Although the assumption of CH can be removed  from  the	 applications of \eqref{Eq.1}  to the Elliott classification programme (\cite[Theorem~A]{farah2020between}), the 
question whether \eqref{Eq.1}  can be proved in  ZFC remained. 

A well-known instance of the  trivializing effect of CH  
is Pa\-ro\-vi\-\v cenko's theorem from general topology. 
Stated in the dual, Boolean-algebraic, form, it asserts that CH implies that all atomless, countably saturated, Boolean algebras of cardinality $\fc$  
are isomorphic.  
In  \cite{van1978parovivcenko} it was proved that the conclusion of Parovi\v cenko's theorem  is equivalent to CH.  
An alternative proof of this fact is given by the main result of \cite{Sh:954} (or by \cite{Sh:818}), asserting that if CH fails then there are $2^{\fc}$ nonisomorphic ultrapowers of the countable atomless Boolean algebra associated with nonprincipal ultrafilters on $\bbN$. By \L o\'s's Theorem all of these Boolean algebras are elementarily equivalent (which in this case reduces to being  atomless, see \cite[Exercise~1.5.3]{ChaKe}), and they are countably saturated and of cardinality $\fc$, being ultrapowers associated with countably incomplete ultrafilters. Clearly,  at most one of these ultrapowers can be isomorphic to $\cP(\bbN)/\Fin$. Theorem~\ref{1.T.B} and Theorem~\ref{2.T.B} below show that in two of the most popular models of ZFC in which CH fails (models of  forcing axioms and Cohen's original model of ZFC),   none of these ultrapowers  is isomorphic to $\cP(\bbN)/\Fin$. 
 
\begin{theorem} \label{1.T.B} 
The Proper Forcing Axiom, PFA,  implies that  $\cP(\bbN)/\Fin$ is not isomorphic to an ultraproduct of  Boolean algebras associated with a nonprincipal ultrafilter on $\bbN$.  
\end{theorem}

Each of our other main results applies to  wider class of structures.  The first one is primarily concerned with a larger class of quotient algebras $\cP(\bbN)/\cI$ associated with Borel ideals on $\bbN$ in place of $\cP(\bbN)/\Fin$. 
 The study of quotient Boolean algebras of the form $\cP(\bbN)/\cI$ for an ideal $\cI$ on $\bbN$, dates back at least to Erd\"os and Ulam (see \cite{erdos2015my}).  The space $\cP(\bbN)$ is identified with the Cantor space, and thus equipped with a canonical compact metrizable topology. If~$\cI$ includes the ideal $\Fin$ of finite  subsets of $\bbN$, then $\cP(\bbN)/\cI$ is atomless and therefore elementarily equivalent to $\cP(\bbN)/\Fin$. We will consider only ideals~$\cI$ that include $\Fin$. 
 
 The question of countable saturation of  $\cP(\bbN)/\cI$ is a bit subtler.  In  \cite{JustKr} it was proved that $\cP(\bbN)/\cI$ is countably saturated for every~$F_\sigma$ ideal $\cI$ that includes  $\Fin$.   The essence of the Just--Krawczyk construction is encapsulated in the concept of a \emph{layered ideal}  in~\cite{Fa:CH}, where it was proved that if $\cI$ is a layered ideal that includes $\Fin$  then $\cP(\bbN)/\Fin$ is countably saturated. (This class of ideals properly includes that of the $F_\sigma$ ideals; for example, for any additively indecomposable countable ordinal $\alpha$ the ideal $\{X\subseteq \alpha\vert $ the order type of~$X$ is less than $\alpha\}$ is layered.)  Since the quotient over an analytic ideal that includes $\Fin$   necessarily has cardinality $\fc$, these results show that all quotients over layered analytic ideals that include $\Fin$ are isomorphic under CH. This conclusion extends to reduced products of countable Boolean algebras, $\prod_\cI A_n$,  associated with layered ideals (\cite[Theorem~2.7]{Sh:1042}). 

In \cite[Corollary 1.4]{just1987separation} it was proved that if an ideal $\cI$ includes $\Fin$  then $\cP(\bbN)/\cI$ is countably saturated if and only if  $\cP(\bbN)/\cI$ has no $(\aleph_0,\aleph_0)$-gaps. In \cite{pacholski1970countably} it was proved that  $\cF$ on $\bbN$ has the property that every reduced product of the form $\prod_n A_n/\cF$ is countably saturated if and only if the Boolean algebra $\cP(\bbN)/\cF$ is countably saturated. Results analogous to the latter one have been  established in  \cite{pacholski1971countably} for  filters on $\aleph_1$  and in \cite{Sh:17} for filters on an arbitrary cardinal.

Note that $\cP(\bbN)/\cI$ is canonically isomorphic to the reduced product, $\prod_\cI A$, where $A$ is taken to be the 2-element Boolean algebra. Also note that all of these reduced products are \emph{projectively definable} in the sense that there is an $n\in \bbN$ and $\mathbf\Sigma^1_n$-formulas $\varphi$,   $\varphi_\wedge$, $\varphi_\vee$, and $\varphi_{\setminus}$, such that the set $\{x\in \bbR\vert \varphi (x)\}$ equipped with the operations defined by $\varphi_\wedge$, $\varphi_\vee$, and $\varphi_{\setminus}$ is a Boolean algebra.\footnote{We avoid using the simpler term projective Boolean algebras in order to avoid confusion with Boolean algebras that are projective objects in the category of Boolean algebras.}

\begin{theorem}\label{2.T.B}
  In a model obtained by adding at least $\fc^+$ Cohen reals to a model of ZFC the following holds. If $\fB$ is a projectively definable Boolean algebra   then  $\fB$ is 
 not isomorphic to an ultraproduct  of countable Boolean algebras associated with a nonprincipal ultrafilter on $\bbN$.
 
 In particular, for any analytic ideal $\cI$ on $\bbN$, the quotient $\cP(\bbN)/\cI$ is not isomorphic to an ultraproduct  of countable Boolean algebras associated with a nonprincipal ultrafilter on $\bbN$.   
\end{theorem}

The proof of Theorem~\ref{2.T.B} applies to a wider class of models, but not to the Sacks models; see \S\ref{S.Concluding}. 

The next result applies to a yet wider range of structures. For the  order property (OP)  and the robust order property see Definition~\ref{Def.T.leqp}. Any theory in which the order property is witnessed by an atomic formula has the robust order property. In particular, the classes of atomless Boolean algebras and dense linear orderings have the robust OP. Another relevant property specific to the theory of  atomless Boolean algebras is that  it is preserved by taking arbitrary reduced products. (Theories with this property are axiomatized by Horn sentences, see e.g., \cite[Theorem 6.2.5']{ChaKe}.) This is not shared by many other theories (e.g., linear orders or fields), and it is because of this fact that in the following result the theories of $A_n$ and $B_n$ are a priori unrelated.

\begin{theorem} \label{T.A}
There exists a forcing extension in which for every countable theory~$T$ that has the robust order property the following holds. 

For every sequence  $(A_n)$ of countable structures in the language of $T$, every sequence $(B_n)$ of countable models of~$T$, and every nonprincipal ultrafilter $\cU$ on~$\bbN$, the following are true.  
\begin{enumerate}
\item \label{1.T.A} The ultraproduct $\prod_{\cU} B_n$ is not isomorphic to $\prod_{\Fin} A_n$ or even to an elementary submodel thereof. 	
\item \label{3.T.A} If the order property of $T$ is witnessed  by a quantifier-free formula then $\prod_{\cU} B_n$  does not embed into $\prod_{\Fin} A_n$. 
\end{enumerate}
\end{theorem}

Since the original impetus for these results drew from the Elliott classification program of \cstar-algebras, we'll explicitly state the relevant corollary. 
If~$A$ is a \cstar-algebra, then the structure $C(K,A)$ as in \eqref{Eq.CKA} is isomorphic to the tensor product $A\otimes C(K)$,  where $C(K)$ is the algebra of continuous complex-valued functions on $K$.  By \cite[Theorem~E]{farah2020between}, for a separable \cstar-algebra~$A$ and a nonprincipal ultrafilter $\cU$ on $\bbN$, CH implies that the ultrapower $(A\otimes C(K))^{\cU}$  is isomorphic to $A^\infty:=\ell_\infty(A)/c_0(A)$ (the latter algebra, known as the asymptotic sequence algebra, is the continuous analog of the reduced power $\prod_{\Fin} A$)  and the   isomorphism extends the identity on $A$ ($A$ is routinely identified with its diagonal copies in $A^{\cU}$ and $A^\infty$).

\begin{corollary}\label{C.D}
	There exists a forcing extension in which the following holds for  every separable \cstar-algebra $A$ and every ultrafilter $\cU$ on $\bbN$. 
\begin{enumerate}
\item	 $(A\otimes C(K))^{\cU} $ is not isomorphic to $A^\infty$. 
\item  $(A\otimes C(K))^{\cU}$ is not  isomorphic to a \cstar-subalgebra of $B^\infty$ for any separable \cstar-algebra $B$.  
	\end{enumerate}
\end{corollary}

The related conclusion,  that $C(K)^{\cU}$ is not isomorphic to a \cstar-subalgebra of $\ell_\infty/c_0$,  is known to be relatively consistent with ZFC and its variant (known as \emph{Woodin's condition}) plays an important role in Woodin's proof of automatic continuity for homomorphisms of Banach algebras (\cite{DaWo:Introduction}).

 The question of the existence of a universal structure among the ultrapowers of a fixed countable (or separable metric) structure is  closely related to the questions answered in theorems stated above. Partial answers to this question (which are easy consequences of the earlier  work of one of the authors) are given in the brief \S\ref{S.Universal}.

\subsection*{Notation} We write
$
X\Subset Y
$
as a shorthand for  `$X\subseteq Y$  and $X$ is finite' (this is sometimes denoted by the formula $X\in [Y]^{<\aleph_0}$). The ideal of finite subsets of a set $X$ is denoted $\Fin_X$ (some authors prefer $[X]^{<\aleph_0}$). For the Fr\'echet ideal $\Fin_{\aleph_0}$ we write $\Fin$. 

\subsection*{Rough outline}
Our proofs use model theory (\S\ref{S.OP}, \S\ref{S.Universal}) and set theory (\S\ref{S.Posets}, \S\ref{S.forcing}). 
In \S\ref{S.OP} we discuss the order property (OP) of first-order theories, discrete and continuous. Standard facts about partially ordered sets (posets) are recalled in  \S\ref{S.Posets}, and depletions of posets are introduced and studied  in~\S\ref{S.Depletion}.  In \S\ref{S.forcing} we define a functor $E\mapsto \bbH_E$ from the category of partial orderings into the category of forcing notions. The material from \S\ref{S.Depletion} is used to prove that the $\bbH_E$ forces that~$E$ embeds into the reduced product $\prod_{n\in \bbN} (n,<)$ ($n$ is identified with $\{0,\dots, n-1\}$)  in a particularly gentle way. Theorem~\ref{T.A} and Corollary~\ref{C.D} are proved in  \S\ref{S.proof.T.A}, while  Theorem~\ref{1.T.B} and Theorem~\ref{2.T.B} are proved in  \S\ref{S.Tie}. 
   In~\S\ref{S.Universal} we record results about the existence of a universal model among the ultrapowers of countable models of $T$ associated with ultrafilters on $\bbN$. Some concluding remarks and questions can be found in~\S\ref{S.Concluding}.

   \subsection*{Acknowledgments} I.F. would like to thank Alan Dow for pointing our attention to \cite{van1978parovivcenko}. The anonymous referee bravely read through the original version of this paper and provided a very helpful report.  For this we are indebted to the referee,   and so should be all future readers of this paper.  
   
\section{Reduced products, the order property, continuous logic} 
\label{S.OP}
In this section we recall the pertinent definitions. It should be emphasized that the first-order theory $T$ is not assumed to be complete.

\subsection{Reduced products} 
We will use the following convention. Suppose that $A_n$ are structures of the same countable language, $\bar a_n$ is a tuple in $A_n$ for all $n\in \bbN$,  and all of these tuples are of the same sort. (If the language is single-sorted, then the sort of a tuple is simply its arity. Note that the natural languages associated with the unbounded metric structures, such as \cstar-algebras, are multisorted, see \cite{Muenster}.)  Then $\bar a$ denotes  the tuple $(\bar a_n)$ in  $\prod_n  A_n$ of the same sort.

If $\cF$ is a filter on $\bbN$ and $A_n$, for $n\in \bbN$, are structures of the same language~$\calL$, then the reduced product $\prod_\cF A_n$ is defined as follows. Its domain is the quotient of $\prod_n A_n$ over the relation  $\bar a\sim_\cF \bar b $ if $\{n\vert a_n =b_n\}\in \cF$.  
The function symbols in $\calL$ are interpreted in the natural way (note that $\sim_\cF$ is a congruence). If $k\geq 1$ and  $R(x(0), \dots, x(k-1))$ is a $k$-ary relation symbol and $\bar a(0),\dots, \bar a(k-1)$ is a $k$-tuple, then we let $\prod_{\cF} A_n\models R(\bar a(0), \dots, \bar a(k-1))$ if and only if the set 
\[
\{n\vert A_n\models R(a_n(0), \dots, a_n(k-1))\} 
\]
belongs to $\cF$.  The image of $\bar a$ in the reduced product $\prod_{\cF} A_n$ under the quotient map is also denoted $\bar a$, by a standard and innocuous abuse of notation. 

If $\cF$ is the Fr\'echet filter (i.e., the filter of cofinite subsets of $\bbN$), then $\prod_{\cF} A_n$ is denoted $\prod_{\Fin} A_n$. (This is yet another standard and innocuous abuse of notation; $\Fin$ denotes the ideal dual to the Fr\'echet filter, and the reduced products are sometimes defined with respect to the dual ideals.)
If~$\cU$ is an ultrafilter (i.e., a proper filter maximal with respect to the inclusion), then $\prod_{\cU} A_n$ is called the ultraproduct. 

When all structures $A_n$ are equal to some $A$, the corresponding reduced products (ultraproducts) are called reduced powers (ultrapowers). 

\begin{remark} \label{Rem.ideal.filter} Reduced products are sometimes defined with respect to a filter $\cF$, and sometimes with respect to the dual ideal $\cF^*$, as convenient. The only difference is in the notation. In this paper we will use the two notions---reduced products with respect to filters and reduced products with respect to ideals---interchangeably. 
\end{remark}

\subsection{The order properties} This combinatorial property  of a first-order theory marks the watershed between well-behaved and wild (see \cite{Sh:c}). 

\begin{definition} \label{Def.T.leqp} Suppose that $T$ is a first-order theory, not necessarily complete. 
\begin{enumerate}
\item 	If $\varphi(\bar x,\bar y)$ is an asymmetric  formula (with $\bar x$ and $\bar y$ of the same sort) in the language of $T$ consider the asymmetric binary relation $\precp$ on a model $A$ of $T$, defined by 
\begin{equation}
\bar a\precp \bar b\text{ if }A\models \varphi(\bar a, \bar b). 
\end{equation}
Some $\bar a_j$, for $j<n$, in $A$ form a  \emph{$\precp$-chain} if for all $i\neq j$ we have $\bar a_i\precp \bar a_j$ if and only if~$i<j$. 
\item If every model of $T$ has an arbitrarily long finite $\precp$-chain, we say that the pair \emph{$(T,\varphi)$ has 
the order property, OP} (\cite{Sh:c}).\footnote{One  says that $\varphi$ has the order property when $T$ is clear from the context.} 
\item The pair $(T,\varphi)$ has the \emph{robust order property} if it has the order property and in addition for models $A_n$, for $n\in \bbN$, of $T$ and $\bar a$ and~$\bar b$ in $\prod_{\Fin} A_n$ we have $\prod_{\Fin} A_n\models \varphi(\bar a, \bar b)$ if and only if the set 
\[
\{n\mid A_n\not\models \varphi(\bar a_n, \bar b_n)\}
\]
is finite. (Note that it is not required that $\prod_{\Fin} A_n$ models  $T$.)
\item The pair $(T,\varphi)$ is said to have the \emph{strict order property} (SOP) if 
the relation $\precp$ is a partial ordering on every model of $T$. 
\end{enumerate}
\end{definition}

The relation between the order property and the robust order property depends on the analysis of the relation between the theories of $A_n$ and the theory of $\prod_{\Fin} A_n$, as given by the Feferman--Vaught theorem (\cite{feferman1959first} and \cite{ghasemi2016reduced} for continuous logic, also see \cite[\S 16.3]{Fa:STCstar}). We will  need only the following easy case.

\begin{lemma} \label{L.atomic} If a pair $(T,\varphi)$ has the order property and $\varphi(\bar x, \bar y)$ is atomic, or a negation of an atomic formula, then the pair $(T,\varphi)$ has the robust order property.  
\end{lemma}

\begin{proof} Fix models $A_n\models T$ for $n\in \bbN$ and suppose  $\varphi$ is an atomic formula. If $\bar a_n$ and $\bar b_n$ are tuples of the appropriate sort in $A_n$ such that $A_n\models \varphi(\bar a_n, \bar b_n)$, then (writing $\bar a$ for the element of the product that has the representing sequence $(\bar a_n)$), we have $\prod_n A_n\models \varphi(\bar a, \bar b)$ and moreover for any filter $\cF$ on $\bbN$ we have $\prod_{\cF} A_n\models\varphi(\bar a, \bar b)$. This also applies to the negation of $\varphi$, and the assertion follows immediately. 
\end{proof}

The proof of Lemma~\ref{L.atomic} uses only the fact that $\varphi$ satisfies the variant of \L o\'s's Theorem for reduced products asserting that $\prod_\cF A_n\models \varphi(\bar a, \bar b)$ if and only if $\{n \mid A_n\models \bar a_n,\bar b_n)\}\in \cF$. A larger class of formulas with this preservation property, called h-formulas, has been isolated in \cite{palyutin1980categorical}.

\subsection{Continuous logic} \label{S.Ctns} 
For more details on continuous logic see \cite{BYBHU} (see \cite{Muenster} or \cite[\S 16]{Fa:STCstar} for operator algebras). That said, this subsection is targeted at the readers already familiar with continuous logic, and its aim is to convince these readers that the proofs of the continuous versions of our main results are analogous to the proofs in the discrete case.  

The value of a formula $\varphi(\bar x)$ evaluated in a model $M$, at a tuple $\bar a$ of the appropriate sort, is denoted $\varphi(\bar a)^M$ and defined by recursion on the complexity of $\varphi$. 

\subsection{Reduced products in continuous logic} If $\cF$ is a filter on $\bbN$, then the reduced product $\prod_{\cF} A_n$ of metric structures of the same language is defined as follows.\footnote{Needless to say, the definition of a reduced product with a filter on some other set is analogous.}  With  $\cF^*:=\{A\subseteq \bbN\mid \bbN\setminus A\notin \cF\}$ (the coideal of all sets positive with respect to $\cF$),  on  $\prod_n A_n$ define pseudometric 
\[
d_\cF(\bar a, \bar b):=\inf_{X\in \cF^*}\sup_{m\in X} d_m(a_m,b_m). 
\]
The universe $A$ of $\prod_{\cF} A_n$ is the completion of the quotient space with respect to $d_\cF$. 

Every predicate symbol $R(\bar x)$ in the language is interpreted as a function into $\bbR$, and the syntax rules of continuous logic require that its interpretation in each $A_n$ respects the same modulus of uniform continuity.  For $\bar a\in \prod_n A_n$ let    
\[
R(\bar a):=\inf_{X\in \cF^*}\sup_{m\in X} R^{A_m}(\bar a_m). 
\]
Then $d_\cF(\bar a, \bar b)=0$ implies $R(\bar a)=R(\bar b)$, and one  defines the interpretation of $R$ in $A$ by $R^A([\bar a]):=R(\bar a)$ (where $[\bar a]$ is the equivalence class of $\bar a$ in $\prod_\cF A_n$).

Function symbols are interpreted in the obvious manner (using the fact that they are  also required to respect the same modulus of uniform continuity in all $A_n$). 

For more details see e.g., \cite[\S 5]{BYBHU} (for ultraproducts) and \cite[\S 16.2 and Definition~D.2.13]{Fa:STCstar} for the general case.

\subsection{Order property in continuous logic}
The following is the continuous analog of  Definition~\ref{Def.T.leqp}. 

\begin{definition} Suppose that  $T$ is a theory in a continuous language, not necessarily complete. 
\begin{enumerate}
\item If $\varphi(\bar x,\bar y)$ is an asymmetric  formula (with $\bar x$ and $\bar y$ of the same sort) in the language of $T$ consider the asymmetric binary relation $\precp$ on a model $A$ of $T$, defined by 
 \begin{align*}
 \bar a\precp \bar b&\text{ if }\varphi(\bar a, \bar b)^A=0\text{ and } \varphi(\bar b, \bar a)^A=1. 
\end{align*}
Some $\bar a_j$, for $j<n$, in $A$ form a  \emph{$\precp$-chain} if for all $i\neq j$ we have $\bar a_i\precp \bar a_j$ if and only if~$i<j$. 

\item  If every model of $T$ has an arbitrarily long finite $\precp$-chain, we say that the pair \emph{$(T,\varphi)$ has 
the order property, OP}  (\cite[Definition~5.2]{FaHaSh:Model2}). 

\item The pair $(T,\varphi)$ has the \emph{robust order property}  if for models~$A_n$, for $n\in \bbN$, of $T$, and all  $\bar a$ and $\bar b$ in $\prod_{\Fin} A_n$  we have 
$\prod_{\Fin} A_n\models \varphi(\bar a, \bar b)$ if and only if for all sufficiently small $\e>0$ the set 
\[ 
\{n\mid  \varphi^{A_n}(\bar a_n, \bar b_n)<\e\text{ and } \varphi^{A_n}(\bar b_n, \bar a_n)>1-\e\}
\]
is finite. (As before,  it is not required that $\prod_{\Fin} A_n$ models  $T$.)
\item The pair $(T,\varphi)$ is said to have the \emph{strict order property} (SOP) if 
the relation $\precp$ is a partial ordering on every model of $T$. 
\end{enumerate}
\end{definition}

Therefore by replacing $\varphi$ with $f(\varphi)$ for a suitable piecewise continuous function $f$,  the  order property of a continuous theory as well as its robustness  are witnessed by a discrete (i.e., 0-1 valued) formula. 
Because of this, we will provide proofs of  our results only in the case of discrete theories, with understanding that they carry on virtually unchanged to the continuous context.  
A proof of the following is analogous to the proof of Lemma~\ref{L.atomic} and therefore omitted. 

\begin{lemma} If $T$ is a continuous theory,  a pair $(T,\varphi)$ has the order property,  and $\varphi(\bar x, \bar y)$ is an atomic formulas then the pair $(T,\varphi)$ has the robust order property.  \qed	\end{lemma}

\section{Background on posets} 
\label{S.Posets}

 In this section we warm up by stating and proving some well-known results. 
 Consider the following two partial quasi-orderings
 on~$\bbNN$ (by $\forall^j$ we denote the quantifier `for all but finitely many $j\in \bbN$'): 
  \begin{align*} 
f\leq^ * g  & \Leftrightarrow (\forall^\infty j) f(j)\leq g(j)\\
f<^ * g  & \Leftrightarrow (\forall^\infty j) f(j)< g(j). 
\end{align*}
Any proper initial segment of  $(\bbNN,\leq^*)$ is included in one of the form $(\{f\in \bbNN\mid f\leq \eta\},\leq^*)$ for some $\eta\in \bbNN$. Such an initial segment is    
isomorphic to   $(\prod_k \eta(k),\leq^*)$ (if $f\leq^* \eta$, then the pointwise minimum of $f$ and $\eta$ is an element of $\prod_k \eta(k)$ equal to $f$ modulo finite) and these structures will be our main focus. 
 The following is essentially a bounded variant of \cite[Proposition~0.1]{Fa:Embedding}. 

\begin{lemma} \label{L.strictly.increasing} There  are $\eta\in \bbN^\bbN$ and
 $
 \Phi\colon (\prod_k k , \leq^*)\to (\prod_k \eta(k), <^*)$ such that for all $f$ and $g$, if $f\leq^* g$ and $g\nleq^* f$ then $\Phi(f)<^* \Phi(g)$.  
  \end{lemma} 
  
A morphism $\Phi$ as guaranteed by Lemma~\ref{L.strictly.increasing} is  called \emph{strictly increasing}. 

\begin{proof} Recursively define $\eta$ by $\eta(0):=1$ and $\eta(n+1):=\sum_{j\leq n} j\eta(j) +1$ for $n\geq 0$. By rewriting the recursive definition of $\eta$, one sees that  for every $n\geq 1$ and every $m\geq 0$ we have 
\begin{equation}\label{Eq.salient}
(m+1)\eta (n)>  \sum_{j<n} j\eta(j) + m\eta (n). 
\end{equation}
Fix $f\in \prod_k k $.  Let $\Phi(f)(0):=0$ and for $n\geq 0$ let 
\[
\Phi(f)(n+1):=\sum_{j\leq  n} f(j) \eta(j). 
\]
Then, since $f(j)<j$ for all $j$, we have  $\Phi(f)(n+1)< \sum_{j\leq n} j\eta(j)<\eta(n+1)$ for all $n\geq 1$, and therefore $\Phi(f)$ belongs to $\prod_k \eta(k)$.

Suppose that   $f$ and $g$ are in $\prod_k k $. 
If $n\geq 1$ is such that $g(n)>f(n)$, then \eqref{Eq.salient} with $m:=f(n)$ implies 
\[
g(n)\eta(n)\geq (f(n)+1)\eta(n)>\sum_{j<n}j\eta(j)+f(n)\eta(n)\geq \Phi(f)(n+1)
\]
and therefore $\Phi(g)(n+1)>\Phi(f)(n+1)$. 
This implies that if  $f$ and $g$ are in $\prod_k k $, then every $n\geq 1$ such that  $f(n)<g(n)$ in addition  satisfies $\Phi(f)(n+1)<\Phi(g)(n+1)$. 

It remains to prove that  $f\leq^* g$ and $g\not\leq^* f$ together imply  $f(k)<g(k)$ for all sufficiently large $k$. Fix such $f$ and $g$, and let $n\geq 1$ be such that $f(n)<g(n)$ and $f(k)\leq g(k)$ for all $k\geq n$. 
Then, as we have just seen, $\Phi(f)(n+1)<\Phi(g)(n+1)$. 
Since
$\Phi(f)(k+1)=\Phi(f)(k)+ f(k)\eta(k)$, for every $k\in \bbN$ the conditions $\Phi(f)(k)<\Phi(g)(k)$ and $f(k+1)\leq g(k+1)$ together imply that $\Phi(f)(k+1)<\Phi(g)(k+1)$. By induction on $k\geq n+1$ one proves that $\Phi(f)(k)<\Phi(g)(k)$ for all $k\geq n+1$, as required.  
  \end{proof}

 The universal structure obtained in Lemma~\ref{L.universal} below is very similar to the  Rado graph, also known as the (countably infinite) random graph,   and it ought to be well-known. 
It was however easier to include a proof than to look for it in the literature. 

\begin{lemma} \label{L.universal} 
There exists an injectively universal countable structure $(C,\precc)$ with an asymmetric  binary relation $\precc$. 
This universality property is absolute between transitive models of a sufficiently large fragment of $\ZFC$. 
\end{lemma}

 \begin{proof} Let $C:=\bbN$ and define the relation $\precc$ as follows. 
 If $m<n$ are in $\bbN$ and $n=\sum_j d_j(n) 3^j$ is the ternary expansion of $n$ (so that $d_j(n)\in \{0,1,2\}$ for all $j$)  then let 
 $m\precc n$ if $d_m(n)=1$,  $n\precc m$ if $d_m(n)=2$, and let $m$ and $n$ unrelated if $d_m(n)=0$. 
 The structure  $(C,\precc)$ has the following property resembling the random graph: 
 \begin{enumerate}
 	 \item [(*)] If $F$ and $G$,  are disjoint finite subsets of $C$, then there exists  $n\in C$ such that 
 $m\in F$ implies $m\precc n$, $m\in G$ implies $n\precc m$, and $m\notin F\cup G$ implies that $m$ and $n$ are unrelated.
 \end{enumerate}
 To see this, let $n:=\sum_{m\in F} 3^m+\sum_{m\in G} 2\cdot 3^m$.

 Given the property (*) of $(C,\precc)$, every countable $(A,\precc')$ can be isomorphically embedded into $(C,\precc)$ by recursion. 
Since  (*) is a first-order property, it is absolute between transitive models of a sufficiently large fragment of $\ZFC$   (see e.g., \cite[Lemma II.4.3]{Ku:Set}). 
   \end{proof}

\section{The depletion of a poset}
\label{S.Depletion}

The notion of the \emph{depletion} of a linear ordering given in Definition~\ref{Def.depletion} appears implicitly in \cite{Fa:Embedding}.   

\begin{definition} \label{Def.depletion}\label{Def.llS}
Suppose that $I$ is a finite linear ordering whose elements are listed in the increasing order as $\xi(i)$, for $i<m$ and $m=|I|\geq 2$.\footnote{If this appears excessively pedantic, note that the conditions \eqref{2.2.depletion} and \eqref{3.2.depletion} below are sensitive to leaving gaps in $I$. We note that $I$ is not necessarily a set of ordinals.}  Also suppose $A$ and $F(\xi)$, for $\xi\in I$, are disjoint sets, and 
$\leq$ is a partial ordering on a set that includes  $E:=A\cup \bigcup_{\xi\in I} F(\xi)$.  Define a binary relation $\ll_I$ on~$E$ as follows.

If $x$ and $y$ belong to $E$, we let $x\ll_I y$ if and only $x\leq y$ and in addition  one of the following applies. 
\begin{enumerate}
\item \label{1.depletion} Both $x$ and $y$ belong to $A\cup F(\xi)$ for some $\xi\in I$.  
\item \label{2.depletion} There are $i<j$ such that  $x\in F(\xi(i))$ and $y\in F(\xi(j))$ and  one of the following holds.  
\begin{enumerate}
\item \label{2.1.depletion} There exists $a\in A$ such that $x\leq a$ and $a\leq y$ 	
\item\label{2.2.depletion}  With $k=j-i$,  there are $x_l\in F(\xi(i+l))$ for $0\leq l \leq k$ such that $x_0=x$, $x_k=y$, and $x_l\leq x_{l+1}$ for all $l<k$.
\end{enumerate}
\item \label{3.depletion} There are $i>j$ such that  $x\in F(\xi(i))$ and $y\in F(\xi(j))$ and  one of the following holds.  
\begin{enumerate}
\item There exists $a\in A$ such that $x\leq a$ and $a\leq y$ 	
\item \label{3.2.depletion} With $k=i-j$,  there are $x_l\in F(\xi(j+l))$ for $0\leq l \leq k$ such that $x_0=y$, $x_k=x$, and $x_l\geq x_{l+1}$ for all $l<k$.  
\end{enumerate}
\end{enumerate}
 The elements $x_i$, for $i\leq k$ as in \eqref{2.2.depletion} or \eqref{3.2.depletion} comprise an \emph{$I$-walk between $x$ and $y$}, or an \emph{$I$-walk with endpoints $x$ and $y$}. The relation $\ll_I$ is called the \emph{depletion of $\leq$ given by $I$, $A$, and $(F(\xi)\vert \xi\in I )$}. When $I$ is clear from the context, we write $\ll$ for $\ll_I$.
\end{definition}

Take note of the fact that a depletion depends on the order on $I$ in an essential way. In addition, it clearly depends on $A$,  the choice of $F(\xi)$ for $\xi\in I$, and the order on $A\cup\bigcup_{\xi\in I} F(\xi)$. Hence writing  $\ll_{(A,(F(\xi)\vert \xi\in I),\leq)}$ in place of $\ll_I$ may appear to be a more reasonable choice. Fortunately,  $A$ and $F(\xi)$, for $\xi$ in some set including $I$,  as well as the ordering $\leq$ on some set that includes $A\cup\bigcup_\xi F(\xi)$, will (unlike $I$) be fixed and clear from the context in any given instance in which a depletion is  used.

It should be emphasized that in both \eqref{2.2.depletion}  and \eqref{3.2.depletion}  it is required that the `walk' between $x$ and $y$ hits $F(\xi(k))$ for all $i\leq k\leq j$. 
%

In order to help the reader internalize these definitions, we state two lemmas whose easy proofs are omitted. 

\begin{lemma} \label{L.two}
	With the notation as in Definition~\ref{Def.llS}, if $|I|\leq 2$, then $\ll_I$ agrees with the restriction of $\leq$ to $A\cup\bigcup_{\xi\in I} F(\xi)$. \qed 
\end{lemma}

\begin{lemma} \label{L.walk} With the notation as in Definition~\ref{Def.llS}, if $s\subseteq t\subseteq  I$,  if  $\{\min(s),\max(s)\}=\{\min(t),\max(t)\}=\{\xi,\eta\}$,  and  $x(\zeta)$, for $\zeta\in t$, is a $t$-walk, then $x(\zeta)$, for $\zeta\in s$, is an $s$-walk with the same endpoints.

On the other hand, it is possible that $\xi(\zeta)$, for $\zeta\in s$, is an $s$-walk but there is no $t$-walk with the same endpoints that extends it. \qed 
\end{lemma}

 The (admittedly rather dull)  Lemma~\ref{L.depletion}  will be instrumental in a critical place in the proof of Theorem~\ref{T.no.kappa.chains}.

\begin{lemma} \label{L.depletion} Suppose that $I$, $A$,  $F(\xi)$, for $\xi\in I$, and an ordering $\leq$ on $A\cup \bigcup_{\xi\in I} F(\xi)$ are as in Definition~\ref{Def.depletion}.  The depletion $\ll$ of  $\leq$ given  by these parameters is a partial ordering whose graph is included in the graph of  $\leq$.
\end{lemma}

\begin{proof} Fix $A$, $I$ enumerated in the increasing order as $\xi(i)$, for $i<m$,    $F(\xi)$, for $i<m$, and an ordering $\leq$ on $E:=A\cup \bigcup_{i<m} F(\xi(i))$. 
	 
We will write $\ll$ for $\ll_s$.	 It is clear from the definition that $x\ll y$ implies $x\leq y$ and that $\ll$ and $\leq$ agree on $A\cup F(i)$ for every $i$. Therefore $\ll$ is antisymmetric and reflexive, and it will suffice to prove that it is transitive. 
	 
	 Towards this end, fix $x,y$, and $z$ such that $x\ll y$ and $y\ll z$. Then $x\leq y$ and $y\leq z$, 
	 and therefore $x\leq z$.  
	 If $x$ and $z$ belong to $A\cup F(\xi(i))$ for some $i$, then $x\ll z$ by \eqref{1.depletion}. 
Therefore if at least one of $x\in A$ or $z\in A$ holds then $x\ll z$, and we may assume 
\begin{enumerate}
	\item [($xz$)]  $x\in F(\xi(i))$ and $z\in F(\xi(j))$ for some distinct $i$ and $j$.
\end{enumerate} 
If $y\in A$ then $x\ll z$ by \eqref{2.1.depletion}. Similarly, if there exists $a\in A$ such that 
$x\leq a$ and $a\leq y$, then $x\ll z$. Also, if there exists $a\in A$ such that $y\leq a$ and $a\leq z$, then $x\ll z$. We can therefore assume (in addition to ($xz$) above) that 
\begin{enumerate}
\item [($y$)] $y\in F(\xi(n))$ for some $n$	
\end{enumerate}
 and that
 both $x\ll y$ and $y\ll z$ are witnessed by instances of \eqref{2.2.depletion}.

 The proof now reduces to the analysis of the ordering of the set $\{i,j,n\}$. The following  claim will help when discussing  the possible cases. 
%

\begin{claim} Suppose that  $i<m$, $0<k\leq m-i$,   $x\in F(\xi(i))$ and $y\in F(\xi(i+k))$. 
\begin{enumerate}
 \item \label{Claim.1} Assume there is no $a\in A$ such that $x\leq a$ and $a\leq y$. 
Then $x\ll y$ if and only if there are $x_l\in F(\xi(i+l))$   for all $0\leq l \leq k$ such that $x_0=x$, $x_k=y$, and $x_l\leq x_{l+1}$ for all $0\leq l<k$. 
\item \label{Claim.2} Assume  there is no $a\in A$ such that $y\leq a$ and $a\leq x$. 
Then $y\ll x$ if and only if there are $x_l\in F(\xi(i+l))$   for all $0\leq l \leq k$ such that $x_0=x$, $x_k=y$, and $x_{l+1}\leq x_{l}$ for all $0\leq l<k$. 
\end{enumerate}
\end{claim} 

  \begin{proof} \eqref{Claim.1} For the direct implication, note that the assumptions imply that~\eqref{2.2.depletion} of Definition~\ref{Def.depletion} applies. 
  Let $x_0:=x$, $x_k:=y$, and for $0<l<k$ let  $x_l$ be a witness for  \eqref{2.2.depletion} of Definition~\ref{Def.depletion}. These objects are clearly as required. 
  
    For the converse implication, 
 assume that $x_l$ for $0\leq l\leq k$ are as in the statement of the claim. Then clearly  \eqref{2.2.depletion} of Definition~\ref{Def.depletion} applies.   

The proof of \eqref{Claim.2} is analogous and therefore omitted.  
\end{proof}

Back to our proof. With $i,j$, and $n$ as in ($xz$) and ($y$), if $i\leq n\leq j$, then part \eqref{Claim.1} of Claim implies that $x\ll z$. 
If $i<j<n$, then the witnessing sequence for $x\ll y$ contains $t\in F(\xi(j))$, such that $x\ll t$ and $t\ll y$. 
But then (since $\ll$ implies $\leq$) $t\leq z$, and $t\ll z$ since both $t$ and $z$ belong to $F(\xi(j))$. 
A proof in the case when $n<i<j$ is similar and uses part \eqref{Claim.2} of the Claim. This proves our claim in the case when $i<j$. 

The proof in the case when $i>j$ is analogous. 
\end{proof}

The following lemma extends Lemma~\ref{L.walk} and stands in contrast to the situation described in Remark~\ref{rem.0}.

\begin{lemma} \label{L.llS} With the notation as in Definition~\ref{Def.llS} and $I$ not necessarily finite, if $s\subseteq t\Subset I$, then for any two $x$ and $y$ in $A\cup \bigcup_{\xi\in s} F(\xi)$ we have that  $x\ll_t y$ implies $x\ll_s y$. 	
\end{lemma}

\begin{proof} Fix $x\ll_t y$ in $A\cup \bigcup_{\xi\in s} F(\xi)$. A glance at Definition~\ref{Def.depletion} shows that we may assume that $x\in F(\xi)$ and $y\in F(\eta)$ for some distinct $\xi$ and $\eta$ in $s$, since in any other situation $x\ll_s y$ follows immediately.  If there is $z\in A$ such that $x\leq z\leq y$, then clearly $x\ll_w y$ whenever $\{\xi,\eta\}\subseteq w\Subset I$. We may therefore assume that there is a $t$-walk with endpoints $x$ and $y$. We denote the `steps' of $s$ by $x(\zeta)$, for $\zeta\in t$ and $\xi\leq \zeta\leq \eta$.  Lemma~\ref{L.walk} implies that  $x(\zeta)$, for $\zeta\in s$ and $\xi\leq \zeta\leq \eta$ is an $s$-walk with endpoints $x$ and $y$, hence $x\ll_s y$. 	
\end{proof}

  \begin{remark} \label{rem.0} Suppose that we are in the situation of Lemma~\ref{L.depletion} and  $s$ is a proper subset of $t$. 
 The fact that in  \eqref{2.depletion}  and \eqref{3.depletion}  the `walk' between $x$ and $y$ is required to hit every $F(\xi(k))$ implies that the graph of $\ll_s$ may be a proper superset of the graph of the restriction of $\ll_t$ to  the domain of $\ll_s$.  
  \end{remark}

Here is another easy lemma with an omitted proof that complements Remark~\ref{rem.0}.

 \begin{lemma}\label{L.interval}
 	Suppose that we are in the situation of Lemma~\ref{L.depletion} and  $s\subseteq t$ is convex (i.e., if $\xi<\eta<\zeta$ are in $t$ and $\{\xi,\zeta\}\subseteq s$ then $\eta\in s$).  Then $\ll_s$ agrees with the restriction of $\ll_t$ to $A\cup\bigcup_{\xi\in I} F(\xi)$. \qed 
 \end{lemma}

 Proposition~\ref{C.314} below is  a relative to a result of Kurepa (\cite{kurepa1948hypothese}) and to   \cite[Theorem~7.1]{Fa:Embedding}. The role that this proposition plays in the proof of our Theorem~\ref{T.no.kappa.chains} is analogous to the role that  \cite[Theorem~7.1]{Fa:Embedding} had played in the proof of  \cite[Theorem~9.1]{Fa:Embedding}.  For reader's convenience, we include a proof. If $I$ is a linear ordering, then $I^*$ denotes the converse ordering.

\begin{prop} \label{C.314}
Suppose that $\kappa$ is an uncountable cardinal, $A$ and $F(\xi)$, for $\xi<\kappa$, are disjoint, and  $\leq$ is a partial ordering of $E:=A\cup\bigcup_{\xi<\kappa} F(\xi)$. In addition suppose that $A$ is countable, all $F(\xi)$ are finite, and $E$ has neither $\kappa$ nor $\kappa^*$-chains. 

Then there exists a cofinal $X\subseteq \kappa$ such that  for any two distinct elements $\xi$ and $\eta$ of $X$ the following condition holds. 
\begin{enumerate}
\item[$*_{(\xi,\eta)}$] there is $s\Subset \kappa$ such that $\{\xi,\eta\}=\{\min(s), \max(s)\}$ and there is no $s$-walk with endpoints in  $F(\xi)$ and $F(\eta)$.\footnote{We emphasize that $s$ is not necessarily a subset of $X$.}  	
\end{enumerate}
\end{prop}

\begin{proof} 
Let $\bbP$ be the poset of all $X\subseteq \kappa$ such that $0\in X$ which satisfy the condition $*_{(\xi,\eta)}$ for all distinct $\xi$ and $\eta$ in $X$, ordered by the inclusion.  This poset is clearly closed under unions of chains, and therefore Zorn's lemma implies that it has a maximal element,  $X$. (Readers not fond of the Axiom of Choice will notice that since $\kappa$ is well-ordered, the existence of $X$ can be proved by recursion in ZF.) We claim that $X$ satisfies the requirements.  
Towards obtaining a contradiction, suppose that $X$ is not cofinal in $\kappa$.

If $X$ has no maximal element, then let $\alpha:=\sup(X)$ and  $Y:=X\cup\{\alpha\}$. In order to verify that $Y$ belongs to $\bbP$, it suffices to verify $*_{(\xi,\alpha)}$ for all $\xi\in X$. Fix $\xi\in X$. We need to prove that there exists $t$ such that there is no $t$-walk with endpoints in  $F(\xi)$ and  $F(\alpha)$.   Let $\eta:=\min(X\setminus (\xi+1))$ and (using $*_{(\xi,\eta)}$)  let $s\Subset \kappa$ be such that $\{\xi,\eta\}=\{\min( s),\max(s)\}$ and there is no $s$-walk with endpoints  $F(\xi)$ and $F(\eta)$. Let $t:=s\cup \{\alpha\}$. Then $\{\xi,\alpha\}=\{\min(t), \max(t)\}$   and Lemma~\ref{L.walk} implies there is no $t$-walk  with endpoints  in $F(\xi)$ and~$F(\alpha)$.

Now suppose that $X$ has a maximal element, $\xi$. Assume for a moment that there are $\eta>\xi$ and $s\Subset \kappa$ which witnesses that $*_{(\xi,\eta)}$ holds. We claim that $X\cup \{\eta\}$ belongs to $\bbP$.  Fix $\zeta\in X$ such that $\zeta<\xi$. Since $X\in \bbP$, there  is $t\Subset \kappa$ such that $\{\zeta,\xi\}=\{\min(t),\max(t)\}$ and  there is no $t$-walk with endpoints in $F(\zeta)$ and $F(\xi)$. With $w:=s\cup t$ we have  $\{\zeta,\eta\}=\{\min(w),\max(w)\}$ and Lemma~\ref{L.walk}  implies that there is no $w$-walk with endpoints in $F(\zeta)$ and $F(\eta)$.  Therefore $*_{(\zeta,\eta)}$ holds. Since $\zeta$ was an arbitrary element of $X$ below~$\xi$ (and since  $*_{(\xi,\eta)}$ holds), this implies  $X\cup \{\eta\}$ is an element of $\bbP$ strictly larger than~$X$; contradiction. 

It will therefore suffice to find $\eta>\xi$ such that $*_{(\xi,\eta)}$ holds.
Assume that such $\eta$ does not exist.  Then for every $\eta>\xi$ and every $s\Subset \kappa$ which satisfies $\{\xi,\eta\}=\{\min(s),\max(s)\}$ there is an $s$-walk with endpoints in  $F(\xi)$ and $F(\eta)$. For each $s$ fix a walk, $x(s,\zeta)$, for $\zeta\in s$, with endpoints in $F(\xi)=\min(s)$ and $F(\eta)=\max(s)$.

Let $\cU$ be an ultrafilter on $\Fin_\kappa$ which for every $s\Subset \kappa$ includes the set $\{t\Subset \kappa\vert s\subset t\}$  (such $\cU$ exists, since the family of  sets of this form has the finite intersection property).  Fix $\xi\leq \zeta<\kappa$.  Since $F(\zeta)$ is finite, there exists a unique $y(\zeta) \in F(\zeta)$ such that $\{s\vert x(s,\zeta)=y(\zeta)\}\in \cU$.  Clearly $y(\zeta)$, for $\xi\leq \zeta<\kappa$, is  a $\kappa$-chain or a $\kappa^*$-chain (as decided by $\cU$); contradiction. 

Therefore there exist $\eta>\xi$ and $s\Subset \kappa$ such that  $\{\xi,\eta\}=\{\min(s),\max(s)\}$ such that there is no $s$-walk with endpoints in  $F(\xi)$ and $F(\eta)$. As already pointed out, this implies $X\cup \{\eta\}\in \bbP$; contradiction.  

Therefore any maximal element  of $\bbP$ is cofinal in $\kappa$, as required. 
	\end{proof}

\section{Gently embedding posets into reduced powers}
\label{S.forcing} 

In the present section we assume that the reader is familiar with the basics of forcing as presented in e.g., \cite{Ku:Set} or \cite{Sh:f}.  
The present section is largely based on \cite{Fa:Embedding}, 
and Theorem~\ref{T.no.kappa.chains} is a close relative to \cite[Theorem~9.1]{Fa:Embedding}.

The category of partially ordered sets is considered with respect to the order-embeddings, i.e., injections  $f\colon E\to E'$ such that $a\leqE b$ if and only if $f(a)\leqEp f(b)$. The category of forcing notions is considered with respect to regular embeddings (also known as complete embeddings, \cite[Definition~III.3.65]{Ku:Set}). If a forcing notion $\bbH_0$ is a regular subordering of a forcing notion $\bbH_1$,  we then write $\bbH_0\lessdot \bbH_1$. Notably, $\bbH_0\lessdot \bbH_1$ is equivalent to the assertion that for every generic filter $G\subseteq\bbH_1$, $G\cap \bbH_0$ is also generic. In other words, $\bbH_1$ can be considered as a two-step iteration of $\bbH_0$ followed by the naturally defined quotient forcing notion.  

If $\kappa$ is an uncountable cardinal, a forcing notion $\bbP$  is said to have \emph{precaliber~$\kappa$} if every set of $\kappa$ conditions in $\bbP$ has a subset of cardinality $\kappa$ such that each of its finite subsets has a common lower bound.  Precaliber $\aleph_1$ is  a strong form of the countable chain condition. 
For example, if $\bbP$ has precaliber $\aleph_1$ then it is \emph{productively ccc}, in the sense that the product of $\bbP$ with any ccc poset is ccc. (We will not need this fact.)

\begin{thm} \label{T.HE} There is  a functor from the category of partially ordered sets 
into the category of forcing notions $E\mapsto \bbH_E$ with the following properties. 
\begin{enumerate}
\item\label{1.T.HE}  $\bbH_E$ has precaliber $\kappa$ for every uncountable regular cardinal $\kappa$. 
\item \label{2.T.HE} $\bbH_E$ forces that there is an  embedding $\Upsilon\colon E \to (\prod_k k , \leq^*)$ (thus for all $a$ and $b$ in $E$ we have  $a\leqE b$ if and only if $\Upsilon(a)\leq^* \Upsilon(b)$).  
\item \label{3.T.HE} If $\kappa>\fc$ is a regular cardinal and neither $\kappa$ nor its reverse $\kappa^*$ embed into $E$, then $\bbH_E$
forces that $\kappa$ does not embed into $\prod_{\Fin} (A_n,\precc)$ for every sequence $(A_n,\precc_n)$ of countable structures equipped with an asymmetric binary relation. 
\end{enumerate}
\end{thm}

\begin{proof} The proof of this theorem will occupy most of the present section. 
For $\bbH_E$ see Definition~\ref{Def.HE}, 
	\eqref{1.T.HE} is Lemma~\ref{L.precaliber}, 
	 \eqref{2.T.HE} is Lemma~\ref{L.generic}, and \eqref{3.T.HE}  is Theorem~\ref{T.no.kappa.chains}. 
\end{proof}

In the Definition~\ref{Def.HE} and elsewhere, if $\dom(f)\subseteq \bbN$ then $f\rs m$ denotes the restriction of $f$ to $m=\{0,\dots, m-1\}$.

\begin{definition}\label{Def.HE}  For a partially ordered set $E$, $\bbH_E$ is the forcing notion defined as follows. 
The conditions of $\bbH_E$ are triples $p=(D_p, n_p, f_p)$, where  	$D_p\Subset E$, $n_p\in \bbN$, 
and $f_p\colon D_p\to \prod_{m<n_p} m$. 

The ordering is defined by letting $p\leq q$  ($p$ \emph{extends} $q$) if the following conditions hold. 
\begin{enumerate}
\item \label{1.Def.HE}	$D_p\supseteq D_q$, $n_p\geq n_q$, $f_p(a)\rs n_q=f_q(a)$ for all $a\in D_q$, and 
\item\label{2.Def.HE}  for all $a$ and $b$ in $D_q$, if $a\leqE b$ then $f_p(a)(j)\leq f_p(b)(j)$
for all $j\in [n_q,n_p)$.  
\end{enumerate}
\end{definition}

In order to relax the notation, if $(p_\xi)$ is an indexed family of conditions in $\bbH_E$ we write $p_\xi=(D_\xi, n_\xi,f_\xi)$. 

\begin{lemma} \label{L.compatible} Suppose that $E$ is a poset,  $R\Subset E$,  $m\geq 2$ and  $p_i$, for $i<m$, are conditions in $\bbH_E$
such that the following holds whenever $i\neq j$. 
\begin{enumerate}
	\item  We have   $D_i\cap D_j=R$. 
	\item  All $a\in R$ satisfy 
$f_i(a)\rs \min(n_i, n_j)= f_j(a)\rs \min(n_i,n_j)$. 
\end{enumerate}
Then some $q\in \bbH_E$ extends all $p_i$.\footnote{We write $q\leq p$ if $q$ extends $p$.}  
\end{lemma}

\begin{proof} Let $D_q:=\bigcup_{i<m} D_i$ and $n_q:=\max_{i<m} n_i$. 
If  $i<m$ is such that $n_i= n_q$, then for  $a\in D_i$ let $f_q(a)=f_i(a)$.
Then $f_q(a)$ is well-defined for $a\in R$ by \eqref{2.Def.HE}. 
For $i<m$  such that $n_i<n_q$ and for  $a\in D_i\setminus R$, let 
(with $\max\emptyset=0$)
\[
f_q(a)(j):=\max\{f_q(b)(j)\mid b\in R, b\leqE a\}. 
\]
for $n_i\leqE j<n_q$. 
This defines $q\in \bbH_E$. We will prove that $q\leq p_i$ for all~$i<m$. 

Clearly, $q$ and $p_i$ satisfy \eqref{1.Def.HE} of Definition~\ref{Def.HE}  for all $i<m$. Fix $i<m$. 

If $n_i=n_q$ then  \eqref{2.Def.HE} of Definition~\ref{Def.HE} is vacuous, hence $q\leqE p_i$. 

Suppose  $n_i<n_q$. To check that $q\leq p_i$, we need to verify \eqref{2.Def.HE} of Definition~\ref{Def.HE}. 
Fix $a$ and $b$ in  $D_i$ such that $a\leqE b$. If there is no $c\in R$ such that $c\leqE b$, then  
 for all $j\in [n_i,n_q)$  we have $f_q(a)(j)=f_q(b)(j)=0$. 
 If there is $c\in R$ such that $c\leqE b$, then $\{c\mid c\leqE a\}\subseteq \{c\mid c\leqE b\}$ and by the definition of $f_q$ we have $f_q(a)(j)\leq f_q(b)(j)$. 

 Thus \eqref{2.Def.HE} of Definition~\ref{Def.HE} holds, and  $q\leq p_i$. 
\end{proof}

\begin{lemma} \label{L.precaliber} 
The poset $\bbH_E$ has precaliber $\kappa$ for every uncountable regular cardinal~$\kappa$.\end{lemma}

\begin{proof} Fix a family $p_\xi$, for $\xi<\kappa$, in $\bbH_E$. 
By the $\Delta$-system lemma and passing to a subfamily of the same cardinality, we may assume that there exists $R\Subset E$ such that 
$D_\xi\cap D_\eta=R$ for all distinct $\xi$ and $\eta$ below $\kappa$. By the pigeonhole principle (using the assumption that $\kappa$ has uncountable cofinality), we may also assume that there exists $n$ such that $n_\xi=n$ for all $\xi$. Also, since there are only finitely many possibilities for $f_\xi(a)$, for $a\in R$, 
we may assume that the functions $f_\xi$ agree on $R$ and therefore we are in the situation of Lemma~\ref{L.compatible}. Therefore, after this refining argument, Lemma~\ref{L.compatible} implies that  every finite subset of $\{p_\xi\mid \xi<\kappa\}$ has a common lower bound. 	\end{proof}

A proof of the  following lemma is straightforward and therefore omitted.  
\begin{lemma} \label{L.dense.sets} For any poset $E$ the following holds.  
\begin{enumerate}
	
\item For every $n$ and every $a\in E$, the set 
\[
\calD(\bbH_E,n,a):=\{p\in \bbH_E\mid n_p\geq n, a\in F_p\}
\]
 is dense in $\bbH_E$. 	
 
\item   If $b\nleqE a$ in $E$, then for every $n\in \bbN$ the set 
 \[
 \cE(\bbH_E, n,a,b):=\{p\in \bbH_E\mid a\in F_p , b\in F_p,\text{ and }  (\exists k\geq n) f_p(a)(k)<f_p(b)(k)\}
 \]
 is dense in $\bbH_E$. 
   \qed
\end{enumerate}
\end{lemma}

\begin{lemma} \label{L.generic} 
If $E$ is a poset and  $G\subseteq \bbH_E$ is a generic filter, then 
\[
\Upsilon_G(a)(j):=f_p(a)(j)
\]
for $p\in G$ defines a strictly increasing function 
$
\Upsilon_G\colon E\to (\prod_k k ,\leq^*)$. 
\end{lemma} 

\begin{proof} By  genericity, $G$ intersects all dense sets defined in  Lemma~\ref{L.dense.sets} and therefore 
 $\Upsilon$ is a strictly increasing map from  $E$ into $(\prod_k k ,\leq^*)$. 
\end{proof}

If $E$ is a subordering of $E'$ then every $p\in \bbH_E$ is (literally) a condition in~$\bbH_{E'}$. 
We will therefore identify $\bbH_E$ with a subordering of $\bbH_{E'}$. 

\begin{lemma} \label{L.regular} If $E'$ is a poset and  $E$ is a subposet of $E'$,  then $\bbH_{E}$ is a regular subordering of $\bbH_{E'}$. 	
\end{lemma}

\begin{proof} The identity map from $\bbH_{E}$ into $\bbH_{E'}$  is clearly an order-embedding. It suffices to prove that there exists a \emph{reduction} (or \emph{projection}) $\pi\colon \bbH_{E'}\to \bbH_{E}$: A map such that  
 for every $p\in \bbH_{E'}$ we have $p\leq  \pi(p)$ and every $q\in \bbH_{E}$ such that $q\leq  \pi(p)$  	is compatible with $p$ (\cite[Lemma~III.3.72]{Ku:Set}). 
Let 
\[
\pi_{E}(p):=(D_p\cap E, n_p, f_p\rs (D_p\cap E)).
\]
Clearly, $p\leq \pi_E (p)$. 
 If $q\leq \pi(p)$, then $D_q\cap D_p=D_p\cap E$ and $f_p(a)(j)=f_q(a)(j)$ for all $a\in D_p\cap D_q$ and all $j<n_p$.  By Lemma~\ref{L.compatible}, $p$ and $q$ are compatible.  
\end{proof}

In the situation when $E$ is a subordering of $E'$, as in Lemma~\ref{L.regular},  
we will need a description of the quotient forcing $\bbH_{E'}/\dot G$, for a generic $G\subseteq \bbH_E$. 
If for some $k\in \bbN$ we have $s\in \prod_{n<k} n$ and $f\in \prod_k k $, then 
\begin{equation*}
	s\sqsubset f\text{ stands for }s=f\rs k. 
\end{equation*}

\begin{definition} \label{Def.quotient} If $E\subseteq E'$ are  partial orderings and 
 $\Upsilon\colon E\to (\prod_k k ,\leq^*)$ is a strictly increasing function,  a forcing notion $\bbH_{E'}(E,\Upsilon)$ is defined as follows. 
The conditions in $\bbH_{E'}(E,\Upsilon)$ are the triples $p=(D_p, n_p, f_p)$, 
where $D_p\Subset E'$, $n_p\in \bbN$, 
 $f_p\colon D_p\to \prod_{j<n} n$, and for $a\in E$ we have $f_p(a)\sqsubset \Upsilon(a)$. 

The ordering is inherited from $\bbH_{E'}$. Therefore  $p\leq q$ ($p$ \emph{extends} $q$)  if the following conditions hold. 
\begin{enumerate}
\item $D_p\supseteq D_q$, $n_p\geq n_q$, $f_p(a)\rs n_q=f_q(a)$ for all $a\in D_q$, and
 \item  for all $a$ and $b$ in $D_q$, if $a\leqEp  b$ then $f_p(a)(j)\leq f_p(b)(j)$
for all $j\in [n_q,n_p)$.  
\end{enumerate}
\end{definition}

Thus $\bbH_{E'}(E,\Upsilon)$ is  a subordering of $\bbH_{E'}$ consisting of those conditions that `agree' with $\Upsilon$ on $E$ and $\bbH_{E'}(E,\Upsilon)$ generically adds an  embedding from $E'\setminus E$ into $\prod_k k $. Note that $\bbH_{E'}(E,\Upsilon)$ is not necessarily separative; this will not cause any issues.

The proofs of the two parts of Lemma~\ref{L.quotient} below are virtually identical to the proofs of \cite[Theorem~4.2]{Fa:Embedding} and \cite[Lemma~4.3]{Fa:Embedding}, respectively.  

\begin{definition}
For a poset $E'$ and  $a\in E'$ let 
 \[
 L(a):=\{b\in E'\mid b\leqEp a\}
 \]
  and 
  \[
  R(a):=\{b\in E'\mid b\geqEp a\}.
  \] 
\end{definition}

\begin{lemma} \label{L.quotient} Suppose $E'$ is a poset, $E$ is a subposet of $E'$, and $\dot G$ is the canonical name for the  $\bbH_E$-generic filter.  
\begin{enumerate}
\item \label{1.L.quotient}	With the projection $\pi_{E}\colon \bbH_{E'}\to \bbH_{E}$ as in the proof of Lemma~\ref{L.regular}, the map $p\mapsto (\pi_{E}(p), p)$ from $\bbH_E$ into $\bbH_{E'}*\bbH_{E'}/\dot G$ is a dense embedding. 
\item  $\bbH_E$ forces that  $\bbH_{E'}/\dot G$ is forcing-equivalent 
to 
$\bbH_{E'}(E,\Upsilon_{\dot G})$ ($\Upsilon_{\dot G}$ is the generic embedding, see Lemma~\ref{L.generic}). 

\item \label{3.L.quotient} If $X\subseteq E$ is such that for every $a\in E'\setminus E$
the set $X\cap L(a)$ is cofinal in $L(a)$ and  
the set $X\cap R(a)$ is coinitial in $R(a)$, then  
$\bbH_E$ forces that
$\bbH_{E'}(E,\Upsilon_{\dot G})$ and 
$\bbH_{E'\setminus (E\setminus X)}(E\cap X,\Upsilon_{\dot G}\rs X)$ are forcing-equivalent.     \qed
\end{enumerate}
\end{lemma}

The following is \cite[Lemma~5.1]{Fa:Embedding} (see also \cite[Lemma~2.5]{brendle1996forcing}). 

\begin{lemma} \label{L.product} Suppose $\bbP_0$ and $\bbP_1$ are forcing notions and $\dot f_j$ is a $\bbP_j$-name for an element of $\prod_k k $ for $j<2$. If 
$\bbP_0\times \bbP_1\forces \dot f_0\leq^* \dot f_1$
then the set of all   $p\in \bbP_0\times \bbP_1$ such that there exist $m\in \bbN$ and $h\in \prod_k k $ which satisfy $p\forces \dot f_0\leq^m \check h$ and $p\forces \check h\leq^m \dot f_1$ is dense in $\bbP_0\times \bbP_1$. \qed
\end{lemma}

In combination with Lemma~\ref{L.product}, the following lemma will be used in a crucial place in the proof of Theorem~\ref{T.no.kappa.chains}. 
 
\begin{lemma} \label{L.HE.product} Suppose $(E,\leq)$ is a poset and $A, B$, and $D$ are subsets of $E$ such that $E=A\cup B$, $D=A\cap B$, and for every $a\in A$ and every  $b\in B$ the following conditions hold.  
\begin{enumerate}
\item $a\leq b$ if and only if $a\leq d$ and $d\leq b$ for some $d\in D$, and
\item   $a\geq b$ if and only if $a\geq d$ and $d\geq b$ for some $d\in D$. 
\end{enumerate}
Then $\bbH_D$ forces that  ($\dot G$ is the canonical name for the generic filter in~$\bbH_D$) $\bbH_E(D,\Upsilon_{\dot G})$ and $\bbH_A(D,\Upsilon_{\dot G})\times 	\bbH_B(D,\Upsilon_{\dot G})$ are forcing equivalent. 
\end{lemma}

With the assumptions of Lemma~\ref{L.HE.product} it can be proved that the function 
\[
\Xi\colon \bbH_E(D,\Upsilon_{\dot G})\to \bbH_A(D,\Upsilon_{\dot G})\times 	\bbH_B(D,\Upsilon_{\dot G})
\]
defined by $\Xi(p):=(\pi_A(p), \pi_B(p))$ is a dense embedding, but we will not need this fact.

\begin{proof} We use  the notation from Lemma~\ref{L.quotient} and write $\dot G(X)$ for the canonical name for the generic filter for $\bbH_X$ (or $\bbH_X(Y,\Upsilon)$ for some $Y$ and $\Upsilon$) where $X$ is $A,B,D$, or $E$.

By Lemma~\ref{L.quotient} \eqref{1.L.quotient}  with $E$ and $A$ in place of $E'$ and $E$, $\bbH_E$ is forcing equivalent to $\bbH_A* \bbH_E(A,\Upsilon_{\dot G(A)})$.  By the same lemma with $A$ and $D$ in place of $E'$ and $E$, $\bbH_A$ is forcing equivalent to $\bbH_D* \bbH_A(D,\Upsilon_{\dot G(D)})$. Therefore $\bbH_E$ is forcing equivalent to the iteration 
\begin{equation}\label{Eq.iteration}
\bbH_D*\bbH_A(D,\Upsilon_{\dot G(D)})*\bbH_E(A,\Upsilon_{\dot G(A)}).
\end{equation}
 The assumptions imply that	$L(b)\cap D$ is cofinal in $L(b)$ and $R(b)\cap D$ is coinitial in $R(b)$, for every $b\in B$. 
Therefore by Lemma~\ref{L.quotient} \eqref{3.L.quotient} applied with  $\bbH_E(A,\Upsilon_{\dot G(D)})$,   $E$, $A$, and $D$ in place of $\bbH_{E'}(E,\Upsilon_{\dot G})$, $E'$, $E$, and~$X$, we conclude that $\bbH_A$ forces that   $\bbH_E(A,\Upsilon_{\dot G(A)})$ is forcing equivalent to $\bbH_B(D,\Upsilon_{\dot G(D)})$ (also recall that $\bbH_D$ is a regular subordering of $\bbH_A$, and that $\Upsilon_{G(A)}$ extends~$\Upsilon_{G(D)}$).   
Since $\bbH_B(D,\Upsilon_{\dot G(D)})$ does not depend on $\dot G(A)$, but only on its intersection with $\bbH_D$, the   iteration in~\eqref{Eq.iteration} is forcing equivalent to 
\begin{equation*}
\bbH_D*(\bbH_B(D,\Upsilon_{\dot G(D)})\times \bbH_A(A,\Upsilon_{\dot G(D)}))
\end{equation*}
and $\bbH_D$ forces that $\bbH_E(A,\Upsilon_{\dot G(A)})$ is forcing equivalent to the product of   $\bbH_B(D,\Upsilon_{\dot G(D)})$ and $\bbH_A(A,\Upsilon_{\dot G(D)})$, as claimed. 
\end{proof}

In the  proof of Theorem~\ref{T.no.kappa.chains} below, for $f$ and $g$ in $C^\bbN$ (with $(C,\precc)$ as guaranteed by   Lemma~\ref{L.universal}) we will write $f\precc^n g$ if $f(j)\precc  g(j)$ for all~$j\geq n$.      
A proof of Theorem~\ref{T.no.kappa.chains} is analogous to, but shorter than,  the proof of \cite[Theorem~9.1]{Fa:Embedding}
(a baroque writeup of this proof with an ample supply of limiting examples and all sorts of   digressions (many  of which were warranted)  can be found in \cite{Fa:Embedding}). 

\begin{thm}\label{T.no.kappa.chains}
		 Suppose $\kappa$ is a regular cardinal such that $\lambda^{\aleph_0}<\kappa$ for every cardinal $\lambda<\kappa$  
		 and $E$ is a partial ordering such that neither $\kappa$ nor $\kappa^*$ embeds into $E$. Then $\bbH_E$ forces that  $\prod_{\Fin} (A_n, \precc_n)$ has no $\kappa$-chains  
		 for any sequence  $(A_n,\precc_n)$, for $n\in \bbN$, 
		  of countable sets with asymmetric binary relations. 
\end{thm}

\begin{proof} 
Since $C$ as in  Lemma~\ref{L.universal} is universal,  $\bbH_E$ forces that $\prod_{\Fin} (A_n, \precc_n)$ has a $\kappa$-chain
for some sequence  $(A_n,\precc_n)$  
(not necessarily in the ground model) if and only if $\bbH_E$ forces 
that $(C^\bbN,\precc^*):=\prod_{\Fin} (C,\precc)$ has a $\kappa$ chain. 
It will therefore suffice to prove the theorem with the additional assumption that $(A_n,\precc_n)=(C,\precc)$ for all $n$.  

  Assume that $\dot f_\xi$, for $\xi<\kappa$, is a name for a $\kappa$-chain in $(C^\bbN,\precc^*)$. 
(We emphasize that this means that for all $\xi<\eta$, $\bbH_E$ forces both 
$\dot f_\xi\precc^* \dot f_\eta$
and $\dot f_\eta\nprecc^* \dot f_\xi$.)
The ccc-ness of $\bbH_E$ implies that for every limit ordinal~$\xi$ there exists a countable $E(\xi)\subseteq E$ such that $\dot f_\xi$, $\dot f_{\xi+1}$, and $\dot f_{\xi+2}$  are $\bbH_{E(\xi)}$-names. Since $\kappa$ is regular  and  $\lambda^{\aleph_0}<\kappa$ for all $\lambda<\kappa$,  the $\Delta$-system lemma for countable sets implies that any family of $\kappa$ countable sets includes a $\Delta$-system of cardinality $\kappa$. By passing to a subfamily, we may assume that the sets $E(\xi)$ form a $\Delta$-system with a countable root $A$. 

For every limit ordinal $\xi$ fix  $q_\xi\in \bbH_E$ 
and $n\in \bbN$ such that 
\begin{equation}\label{eq.qxi}
q_\xi\forces_{\bbH_E}  \dot f_\xi\precc^n  \dot f_{\xi+1}
 \precc^n  \dot f_{\xi+2}.
 \end{equation}
 By the pigeonhole principle and passing to a subfamily if necessary, we may assume that $n$ as in \eqref{eq.qxi} is the same for all $\xi$.
Lemma~\ref{L.regular} implies that~$\bbH_{E(\xi)}$ is a regular subordering of $\bbH_E$. We may therefore assume   that $q_\xi\in \bbH_{E(\xi)}$ and we have  
\[
q_\xi\forces_{ \bbH_{E(\xi)}}  \dot f_\xi\precc^n  \dot f_{\xi+1}
 \precc^n  \dot f_{\xi+2}.
 \]
   Writing $q_\xi=(D_\xi,n_\xi, f_\xi)$, let 
 $
 F(\xi):=D_{\xi}\setminus A$.
Note that each $F(\xi)$ is finite.

 Fix $\xi<\kappa$ for a moment and fix a generic filter  $G\subseteq \bbH_{A\cup F(\xi)}$ such that $q_\xi\in G$.  Recall that the domain of $C$ is $\bbN$, and  is therefore equipped with a well-ordering (as a matter of fact, the first well-ordering known to man).  For $j\in\bbN$ we can  let $h_\xi(j)$ be the least element of $C$ (in this well-ordering) such that 
\[
r\forces \dot f_{\xi+1}(j)=\check c
\]
for some $r$ in the quotient $\bbH_E(A\cup F(\xi),\Upsilon_G)/G$ (see Lemma~\ref{L.quotient}). 

 This defines $h_\xi\in C^\bbN$ in $V[G]$. Use the Maximality Principle (\cite[Theorem~IV.7.1]{Ku:Set}) to choose  an $\bbH_{A\cup F(\xi)}$-name
  $\dot h_\xi$ for this function.

\begin{claim} The condition $q_\xi$ forces that $\dot f_\xi\precc^n \dot h_\xi\precc^n f_{\xi+2}$. 	
\end{claim}

\begin{proof} If there are $r\leq q_\xi$ in $\bbH_E$ and $j\geq n$ such that $r\forces \dot f_\xi(j)\nprecc \dot h_\xi(j)$, fix a generic filter $G$ in $\bbH_E$ containing $r$. Then in $V[G]$ we have $\dot f_\xi\nprecc^n \dot f_{\xi+1}$, although $q_\xi\in G$;  contradiction. An analogous argument gives that there is no $j\geq n$ such that some 
 $r\leqE q_\xi$   forces that $\dot h_\xi(j)\nprecc \dot f_{\xi+2}(j)$. 
	\end{proof}

 The pairs $(q_\xi,\dot h_\xi)$ are indexed by  limit ordinals below $\kappa$. We re-enumerate them preserving the order and obtain conditions $q_\xi$ and names $\dot h_\xi$ for $\xi<\kappa$. 
Since $\bbH_E$ has the ccc, some condition $q\in \bbH_E$ forces that $\kappa$  of the $q_\xi$'s belong to the generic filter. Therefore $q$ forces that the family of all $h_\xi$ such that~$q_\xi$ belongs to the generic filter is a $\kappa$-chain in $(C^\bbN, \precc^*)$. 
%
%

By Proposition~\ref{C.314}, there exists a cofinal $X\subseteq \kappa$ such that for any two distinct elements $\xi<\eta$ of $X$ there is $s\Subset \kappa$ such that  $\{\xi,\eta\}=\{\min(s),\max(s)\} $ and  there is no $s$-walk with endpoints in  $F(\xi)$ and $F(\eta)$.

Fix $\xi\in X$,  let  $\xi':=\min(X\setminus (\xi+1))$, and fix $s\Subset \kappa$ such that  $\{\xi,\xi'\}=\{\min(s),\max(s)\}$ and there is no $s$-walk with endpoints in $F(\xi)$ and  $F(\xi')$.  We will analyze the relation between  the names $\dot h_{\xi}$ and $\dot h_{\xi'}$. 

 Let $\xi(j)$, for $j<n$,\footnote{This $n$ is unrelated to the $n$ appearing in \ref{eq.qxi}. No danger of confusion here.}
 be an increasing enumeration of $s$, so that in particular $\xi(0)=\xi$ and $\xi(n-1)=\xi'$.   Consider the depletion $\ll_s$ of $\leqE$ given by $A$ and $F(\eta)$, for $\eta\in s$.  By Lemma~\ref{L.depletion}, $\ll_s$ is a partial ordering on 
 \[
 A':=A\cup\bigcup_{j<n} F(\xi(j)).
 \]
   For $i<j<n$ let 
\[
A(i,j):=A\cup F(\xi(i))\cup F(\xi(j)),
\]
 ordered by $\ll_s$. Note that if $|i-j|>1$ then $A(i,j)$ is  ordered by a restriction of the depletion associated with $s$ which can differ from the natural depletion $\ll_{(\xi(i), \xi(j))}$ (Remark~\ref{rem.0}). 
 However, for $i<n-1$,  Lemma~\ref{L.interval} implies that the restriction of $\ll_s$ to $A(i,i+1)$ agrees with the depletion $\ll_{\{\xi(i),\xi(i+1)\}}$, which by Lemma~\ref{L.two} agrees with the ordering induced from~$E$. By Lemma~\ref{L.regular}, we have    $\bbH_{A(i,i+1)}\lessdot \bbH_E$. 

Since $A'$  is ordered by $\ll_s$, which possibly disagrees with  the ordering of~$A'$ inherited from $E$,  the posets $\bbH_{A'}$ and $\bbH_E$ are possibly unrelated. On the other hand, each $A(i,j)$ is a subordering of $A'$ and therefore  Lemma~\ref{L.regular} implies   $\bbH_{A(i,j)}\reg \bbH_{A'}$.

Since $\dot h_{\xi(i)}$ and $\dot h_{\xi(j)}$ are $\bbH_{A(i,j)}$-names and $\bbH_E$ forces
$\dot h_{\xi(i)}\precc^* \dot h_{\xi(i+1)}$ for all $i<n-1$, we have that  $\bbH_{A'}$ forces 
$\dot h_{\xi(i)}\precc^* \dot h_{\xi(i+1)}$ for all $i<n-1$. 
 By transitivity, $\bbH_{A'}$ forces 
$\dot h_{\xi(0)}\precc^* \dot h_{\xi(n-1)}$. 
Since $\dot h_{\xi(0)}$ and $\dot h_{\xi(n-1)}$ are $\bbH_{A(0,n-1)}$-names and 
$\bbH_{A(0,n-1)}\reg \bbH_{A'}$,   we have that $\bbH_{A(0,n-1)}$  forces $\dot h_{\xi(0)}\precc^* \dot h_{\xi(n-1)}$.

 Since there is no $s$-walk whose endpoints are some  $x\in F(\xi(0))$ and some $y\in F(\xi(n-1))$, we have that $x\ll_s y$ implies there is $a\in A$ such that $x\leq a\leq y$, and that $ y\ll_s x$ imples there is $a\in A$ such that $y\leq a\leq x$. This means that the assumptions of Lemma~\ref{L.HE.product} with $E$, $A$, $B$, and $D$ replaced with  $A\cup F(\xi(0))\cup F(\xi(n-1))$, $A\cup F(\xi(0))$, $A\cup F(\xi(n-1))$, and~$A$, respectively, (sorry!) are satisfied.  Therefore if $G\subseteq \bbH_A$ is a generic filter then the quotient $\bbH_{A\cup A(0,n-1)}/G$ is forcing-equivalent to the product 
\[
\bbH_{A\cup F(\xi(0)}(A,\Upsilon_{G})
\times \bbH_{A\cup F(\xi(n-1))}(A,\Upsilon_{G}). 
\]
Most importantly, the 
names $\dot h_{\xi(0)}$ and $\dot h_{\xi(n-1)}$ are added by the two factors of this product. By Lemma~\ref{L.product}, there exist a condition $p_\xi\in \bbH_A$ and an $\bbH_A$-name $\dot g_\xi $ (recall that $\xi=\xi(0)$) 
such that 
\[
p_\xi\forces \dot h_{\xi(0)}\leq^* \dot g_\xi\leq^* \dot h_{\xi(n-1)}.
\] 
Since $\bbH_A$ is countable, there is $q\in \bbH_A$ such that $Y=\{\xi\in X\vert p_\xi=q\}$ is a cofinal subset of $X$ (and of $\kappa$).  Therefore $q$ forces that~$\bbH_A$ adds a strictly increasing $\kappa$-chain $\dot g_\xi$, for $\xi<\kappa$, to $(\bbNN,\leq^*)$. Since $A$ is countable, $\bbH_A$ cannot add more than $\fc$ reals; contradiction. 
\end{proof}

The robustness of the robust order property (Definition~\ref{Def.T.leqp}) is used in the following proposition.  

\begin{prop} \label{P.HE.embeds}
	Suppose that the pair $(T,\varphi)$ has the robust order property and $E$ is any poset. Then $\bbH_E$ forces the following. 
\begin{enumerate}
\item  The poset  $E$ embeds into $\prod_{\Fin} (A_n, \precp)$ for every sequence $(A_n)$ of models of $T$. 
\item For any nonprincipal ultrafilter $\cU$ on $\bbN$ there is strictly increasing map from $E$ into $\prod_{\cU} (A_n,\precp)$ whose range is linearly ordered by $\precp$.  
\end{enumerate}
	\end{prop} 

\begin{proof} The first part is almost obvious, but proving it in some detail will also provide a proof of the second part.  

By Lemma~\ref{L.strictly.increasing}, there are $\eta\in \bbNN$ and  $\Phi\colon \prod_{\Fin} (n,\leq^*)\to \prod_{\Fin} (\eta(n),<^*)$ such that $f\leq^* g$ and $g\nleq^* f$ implies $(\forall^\infty n) \Phi(f)(n)<\Phi(g)(n)$. Since $A_n$ is a model of $T$, there exists a $\precp$-chain $C_n$ of length $\eta(n)$ in $A_n$. By  identifying this chain with $(\eta(n),\leq)$, we obtain $\Phi\colon \prod_{\Fin} (n, \leq^*)\to\prod_{\Fin} (A_n,\precp)$ such that $f\leq^* g$ and $g\nleq^* f$ implies  $\Phi(f)(n)\precp \Phi(g)(n)$ and $\Phi(g)(n)\nprecp \Phi(f)(n)$ for all but finitely many $n$. By composing the embedding of $E$ into $\prod_{\Fin} (n,\leq^*)$ provided by Theorem~\ref{T.HE} with $\Phi$,  we obtain  an $\bbH_E$-name  for an embedding of $\Xi\colon E\to \prod_{\Fin} (A_n,\precp)$ (this proves the first part; read on for the proof of the second part) that in addition has the property that $a\lE b$  implies
  \begin{equation*}
  (\forall^\infty n)(\Xi(f)(n)\precp \Xi(g)(n)\text{ and }\Xi(g)(n)\nprecp \Xi(f)(n)).
    \end{equation*}
Let $\cU$ be a nonprincipal ultrafiter on $\bbN$ and let $\pi_{\cU}$ denote the quotient map from $\prod_{\Fin} A_n$ to $\prod_{\cU} A_n$. Then the displayed formula  implies that the restriction of $\pi_{\cU}$ to  $\Xi[E]$ is strictly increasing. The range of this map is the ultraproduct of the $\precp$-chains $C_n$, and therefore linearly ordered by \L o\'s's Theorem.   This proves the second part. 
 \end{proof}

\section{Proofs of Theorem~\ref{T.A} and Corollary~\ref{C.D}}
\label{S.proof.T.A} 

The proof of Theorem~\ref{T.A} will use the following result (see   \cite[Theorem~3.2]{Fa:Embedding} for a proof)	.

\begin{thms} [Galvin] For every uncountable cardinal $\kappa$ there exists a partial ordering $E_\kappa$ such that $E_\kappa$ has no infinite chains but for every linear ordering~$\calL$ such that there are neither $\kappa$-chains nor  $\kappa^*$-chains in $\calL$ there is no strictly  increasing map $\Phi\colon E\to \calL$. \qed 	
\end{thms}

\begin{proof}[Proof of Theorem~\ref{T.A}] 
Fix a theory $T$ with the robust order property. 
We will prove that the Levy collapse of the continuum to~$\aleph_1$ followed by  $\bbH_E$ for $E$ provided by Galvin's theorem forces both conclusions  of  Theorem~\ref{T.A}.  

These proofs  have a common initial segment that we now present. 

In the extension by the Levy collapse of the continuum to $\aleph_1$, let $\kappa>\fc$ be a regular cardinal ($\kappa=\aleph_2$ will do). Let $E$ be the poset as guaranteed by Galvin's theorem stated at the beginning of this section. We will prove that the Levy collapse followed by $\bbH_E$ is the forcing notion as promised in the statement of Theorem~\ref{T.A}.

Fix an ultrafilter $\cU$ on $\bbN$, a  sequence  $(A_n)$ of countable structures in the language of $T$, and a  sequence $(B_n)$ of countable models of~$T$. 	Proposition~\ref{P.HE.embeds} implies that~$\bbH_E$ adds a strictly increasing map from $E$ into $\prod_{\cU} (B_n,\precp)$ whose range is linearly ordered by $\precp$. By the choice of $E$, there exists a $\kappa$-$\precp$-chain or a $\kappa^*$-$\precp$-chain in $\prod_{\cU} (B_n,\precp)$. 

On the other hand, by Theorem~\ref{T.HE}  there are neither $\kappa$-chains nor $\kappa^*$-chains in $\prod_{\Fin} (A_n, \precp)$.

From this point on the proofs of \eqref{1.T.A} and \eqref{3.T.A} differ. 

\eqref{1.T.A} We need to prove that  $\prod_{\cU} B_n$ is not isomorphic to an elementary submodel  of $\prod_{\Fin} A_n$. Since  elementary embeddings  preserve $\precp$, this is immediate from the fact that the former contains a $\kappa$- or $\kappa^*$-$\precp$-chain and the latter does not. 

 \eqref{3.T.A} Suppose in addition that   $\varphi$ is quantifier-free.  Then all embeddings preserve~$\precp$,  and $\prod_{\cU} B_n$ cannot be isomorphic to a submodel of $\prod_{\Fin} A_n$. 
 \end{proof}

\begin{proof}[Proof of Corollary~\ref{C.D}] 
	Suppose that $A$ is a separable \cstar-algebra and $\cU$ is an ultrafilter on $\bbN$. If $\cU$ is principal, then $(A\otimes C(K))^{\cU}$ is isomorphic to $A\otimes C(K)$ while $A^\infty$ is nonseparable.  We  may therefore assume that~$A$ is infinite-dimensional and that $\cU$ is nonprincipal. 
	 
	The theory of infinite-dimensional \cstar-algebras has the order property  witnessed by an atomic formula (\cite[Lemma~5.3]{FaHaSh:Model1}). Therefore the theory of $A\otimes C(K)$ has the robust order property, and  Theorem~\ref{T.A} \eqref{3.T.A} implies that  $(A\otimes C(K))^{\cU}$ does not embed into $B^\infty$ for any \cstar-algebra~$B$.  
\end{proof}

\section{Proofs of Theorem~\ref{1.T.B} and Theorem~\ref{2.T.B}: Tie points} 
\label{S.Tie} 

The contents of this section is rather accurately described  by its title. 

\begin{definition} Suppose $X$ is a compact Hausdorff space.
A point $x\in X$ is a \emph{tie point} if there are closed 
subsets $A$ and $B$ of $X$ such that $A\cup B=X$ and $A\cap B=\{x\}$ (in symbols, $A\bowtie_x B$). 
\end{definition} 

Two subsets $\cI$ and $\cJ$ of a Boolean algebra $\fB$ are \emph{orthogonal} if 
$a\wedge b=0_{\cB}$ for all $a\in \cI$ and all $b\in \cJ$. 
The following is proved by parsing the definitions. 

\begin{prop} \label{P.tie}
Suppose $\fB$ is a Boolean algebra. 
The following are equivalent for 
an ultrafilter $\cU$ on $\fB$. 
 \begin{enumerate}
 \item The complement of $\cU$ is equal to the union of two orthogonal ideals.
 \item \label{2.P.tie} $\cU$ is a tie-point in the Stone space of $\fB$. \qed 
 \end{enumerate}
 \end{prop}
  
\begin{definition}   By analogy with true P-points, 
  an ultrafilter $\cU$ in a Boolean algebra  is called a \emph{true tie point}
  if the ideals as in Proposition~\ref{P.tie} \eqref{2.P.tie} 
  can be chosen so that each one of them is generated by a linearly ordered subset. 
   \end{definition} 
   
   The salient point of the proof of the following is the observation that true tie points are $\Sigma_1$-definable, but the reader may choose to ignore this remark and read the proof instead.   
    
   \begin{lemma}
   	\label{L.truetie}
  Every ultraproduct of countable atomless Boolean algebras  has a true tie point. 
  \end{lemma}
  
  \begin{proof} Every ultrafilter in a  countable atomless Boolean algebra is a true tie point, since the generating sets of order type $\omega$ can be chosen by recursion. Suppose  $\prod_{\cU} C_n$ is an ultraproduct of countable atomless Boolean algebras. If $\cU$ is principal, then $\prod_n C_n$ is isomorphic to one of the $C_n$'s and the assertion follows from the first sentence of this proof. 
  
  Now assume $\cU$ is nonprincipal. For every $C_n$ fix a true tie point $p_n$ and linearly ordered generating sets 
  $\cA_n$ and $\cB_n$ for the ideal $C_n\setminus p_n$. Then $(C_n, \cA_n,\cB_n)$ is an expansion of $C_n$ to the language with two additional unary predicates. Each one of these structures satisfies the following: Both $\cA_n$ and $\cB_n$ are linearly ordered, $A\wedge B=\emptyset$ for all $A\in \cA_n$ and $B\in \cB_n$, and for every $X\in C_n$ either $X$ or its complement belongs to $\cA_n\cup \cB_n$. These are all first-order statements, and they imply that the complement of $\cA_n\cup \cB_n$ is an ultrafilter.  
   
  The ultraproduct $\prod_{\cU} (C_n,\cA_n,\cB_n)$ 
  is an expansion of $\prod_n C_n$ and by \L o\'s's Theorem the sets $\cA:=\prod_{\cU} \cA_n$ and $\cB:=\prod_{\cU} \cB_n$ generate ideals of $\prod_{\cU} C_n$ whose complement is a true tie point. 
   \end{proof}

\begin{proof}[Proof of Theorem~\ref{1.T.B}] 
We need to prove that PFA implies $\cP(\bbN)/\Fin$ is not isomorphic to an ultraproduct of Boolean algebras associated with a nonprincipal ultrafilter on $\bbN$.  By \cite{van1989there} (see
 \cite[Corollary~1.9]{Sh:1057}), PFA implies that there are no tie points in $\cP(\bbN)/\Fin$, while there are tie points in an ultraproduct of countable atomless Boolean algebras by Lemma~\ref{L.truetie}. 
\end{proof} 

The following will be used in the proof of Theorem~\ref{2.T.B}. 
   
   \begin{lemma} \label{L.Cohen}    The poset for adding at least $\fc^+$ Cohen reals forces that  every  projectively definable atomless Boolean algebra $\fB$ has no true tie points. 
   \end{lemma}

\begin{proof} Suppose that  $\kappa\geq \fc^+$  and let~$\bbC_\kappa$ denote the poset for adding $\kappa$ Cohen reals. Fix an $n\in \bbN$ and $\mathbf\Sigma^1_n$-formulas $\varphi$,   $\varphi_\wedge$, $\varphi_\vee$, and $\varphi_{\setminus}$, which define $\fB$. We will only need the $\mathbf \Delta^1_{n+1}$ formula $\varphi_{<}$, that defines the relation $a<b$ in $\fB$.  
 By passing to an intermediate forcing extension, without a loss of generality we may assume that the reals coding these formulas belong to  the ground model. 

 By  genericity, no nonprincipal ultrafilter on $\bbN$ in the forcing extension is generated by fewer than~$\kappa$ subsets of $\bbN$.  (This is well-known, but here is a sketch of the proof: After adding $\kappa$ Cohen reals,  for every $\cX\subseteq \cU$ of cardinality less than $\kappa$ there is a Cohen real $Y$ generic over $V[\cX]$. For every infinite $X\subseteq \bbN$, the set of all $Y\subseteq \bbN$ such that $X\cap U$ and $X\setminus Y$ are both infinite is comeager. Therefore $\cX$ does not `decide' whether $Y\in \cU$ or $\bbN\setminus Y\in \cU$.)
	Assume $p$ is a true tie point in $\fB$ and let $\cA$ and $\cB$ be the linearly ordered (modulo $\cI$) sets whose complements generate $\fB\setminus p$.  By genericity, at least one of $\cA$ and $\cB$ has cofinality greater than $\fc$. 
By interchanging $\cA$ and $\cB$, we may assume  that the cofinality of  $\cA$ is $\kappa>\fc$.

	     The proof is completed by Kunen's isomorphism of names argument (\cite{kunen1969inaccessibility}; see \S\ref{S.Concluding} for disambiguation) that  we now sketch.

Suppose that  $\dot f_\xi$, for $\xi<\kappa$, is a $\bbC_\kappa$-name for a strictly increasing chain cofinal in $\cA$. In particular, $f_\xi$, for $\xi<\fc^+$, is a name for a strictly increasing chain in $\cA/\cI$. This will suffice to obtain a contradiction.  Since each $\dot f_\xi$ is a name for a real, it is coded by a sequence of antichains and the union, denoted $D_\xi$,  of the supports of all conditions in these antichains  is a countable subset of $\kappa$. Since $\fc^+$ is regular and $\lambda<\fc^+$ implies $\lambda^{\aleph_0}<\fc^+$, by passing to    a cofinal subset we may assume that the sets $D_\xi$ form a $\Delta$-system with root~$R$. By  a counting argument and passing to a cofinal subset again we may assume that the restrictions of $\dot f_\xi$ to $R$ agree, and that $\dot f_\xi$ and $\dot f_\eta$ are isomorphic for all $\xi$ and $\eta$.  This means that  for $\xi<\eta$ there is an automorphism $\Phi_{\xi\eta}$ of $\bbC_\kappa$  that sends $\dot f_\xi$ to $\dot f_\eta$, for any two $\xi<\eta<\fc^+$. However, since~$\bbC_\kappa$ forces that $\varphi_<(\dot f_\xi, \dot f_\xi)$ is true and the real coding the asymmetric formula  $\varphi_<$  is in the ground model, this is a contradiction that completes the proof.  
\end{proof}

\begin{proof}[Proof of Theorem~\ref{2.T.B}]
If $\cI$ does not include the Fr\`echet filter then $\cP(\bbN)/\cI$ has atoms, and therefore cannot be isomorphic to $\cP(\bbN)/\Fin$. We may therefore assume that $\cI$ incudes the Fr\`echet filter. 

In the model obtained by adding at least~$\fc^+$ Cohen reals to a model of ZFC, suppose that $\fB$ is a projectively definable Boolean algebra. By Lemma~\ref{L.Cohen}, there are no true tie points in $\fB$. By Lemma~\ref{L.truetie},  in every model of $\ZFC$ there is a true tie point in any  ultraproduct of countable atomless Boolean algebras.   
\end{proof}

\section{The existence of universal ultrapowers} 
\label{S.Universal} 

Hitherto unbeknownst to the junior (!) author, some questions closely related to those resolved in our main results have easy answers, collected in this  section. 

Suppose that $T$ is a theory in a countable (or separable) language and let 
\begin{align*}
\bbM_T&=\{A^\cU\vert A\models T, A\text{ is countable (separable), and }\cU\in \beta\bbN\setminus\bbN\},\\
\MODc(T)&=\{A\vert A\models T, |A|=\fc\}.
\end{align*}
 CH implies that all $M\in \bbM_T$ are saturated, and therefore isomorphic. This conclusion is by \cite[Theorem~5.6]{FaHaSh:Model2} (also  \cite{Sh:954}) equivalent to CH. In some applications it suffices to know that among the ultrapowers of a model~$A$ of $T$ there exists one which is (injectively) universal. We'll say that a set of models has a $\preceq$-\emph{universal element} if it has an element universal under elementary embeddings, and that it has a $\hookrightarrow$-\emph{universal element} if it has an element universal under not necessarily elementary embeddings.

As pointed out in \cite[p. 181]{shelah2021divide}, this sort of nitpicking (distinguishing between $\hookrightarrow$ and $\preceq$) is in general unnecessary, since one can make the difference disappear by expanding the language by Skolem functions. We nevertheless nitpick because the existence of Skolem functions in continuous logic is a delicate problem.   

\begin{prop} \label{P.Universal} Suppose that $T$ is a first-order theory.   Then $\bbM_T$ has a $\preceq$-universal ($\hookrightarrow$-universal) element if and only if $\MODc(T)$ does. 
\end{prop}

\begin{proof} 
It is well-known that every model of $T$ of cardinality $\fc$ is isomorphic to an elementary submodel of an ultrapower of a countable model of~$T$. This follows from the results of \cite[Chapter VI.5]{Sh:c} or \cite{Sh:954}. Therefore any $M\in \bbM_T$ which is $\preceq$-universal is also $\preceq$-universal for $\MODc(T)$, and similarly for $\hookrightarrow$-universality.  
\end{proof}

 \begin{prop} If $T$ is a stable theory in a countable language, then $\bbM_T$  has a $\preceq$-universal, and therefore a $\hookrightarrow$-universal, model (in ZFC). 
  \end{prop}

\begin{proof} In \cite[Theorem~5.6]{FaHaSh:Model2} it was proved that if $T$ is stable then all ultrapowers $A^\cU$ for a countable (or separable) $A$ and $\cU\in \beta\bbN\setminus \bbN$ are saturated, and therefore isomorphic. 
\end{proof}

We can therefore assume that $T$ is not stable, or equivalently, that it has the order property (for the continuous case, this equivalence is in  \cite[Theorem~5.5]{FaHaSh:Model2}, generalizing classical result of the second author, \cite{Sh:c}). 

\begin{quest} \label{Q.Universal} Suppose that $T$ is a theory in a countable language with the order property and CH fails. Can $\MODc(T)$ (equivalently, $\bbM_T$) have a $\preceq$-universal, or a $\hookrightarrow$-universal,  model? 
\end{quest}

We give some partial answers to this question. The strict order property (SOP) of $T$ is the strengthening of the order property in which the witnessing formula~$\varphi$ is required to define a (partial) ordering on every model of $T$. 

\begin{prop} \label{P.SOP} Suppose that $T$ is a  theory in a countable language with the SOP. If there exists a cardinal $\kappa$ such that $\kappa^+<\fc=\cf(\fc)<2^\kappa$ and $T$ is complete, then $\bbM_T$ does not have a $\preceq$-universal model. 

If the SOP is witnessed by a quantifier-free formula, then even if $T$ is not complete,  $\bbM_T$ does not even have a $\hookrightarrow$-universal model. 
\end{prop}

\begin{proof} The existence of $\kappa$ as stated implies there is no $\hookrightarrow$-universal linear order of cardinality $\fc$  by \cite[Theorem~3.10]{Sh:409}. By \cite[Theorem~5.5]{Sh:409}, this implies that $\MODc(T)$ has no $\preceq$-universal element.   As explained in \cite{Sh:409}, if the SOP is witnessed by a quantifier-free formula, then $\MODc(T)$ has no $\hookrightarrow$-universal element. 
\end{proof} 

The SOP$_4$ (\cite[Definition~2.5]{Sh:500}) is a technical weakening of the strict order property, hence the following is a strengthening of Proposition~\ref{P.SOP}. 

\begin{prop}\label{P.SOP4} 
	 Suppose that $T$ is a  theory in a countable language with the SOP$_4$. If there exists a cardinal $\kappa$ such that $\kappa^+<\fc=\cf(\fc)<2^\kappa$ and $T$ is complete, then $\bbM_T$ does not have a $\preceq$-universal model. 

If the SOP$_4$ is witnessed by a quantifier-free formula, then even if $T$ is not complete,  $\bbM_T$ does not even have a $\hookrightarrow$-universal model. 
\end{prop}

 \begin{proof} This is \cite[Theorem~2.12]{Sh:500}. 
  \end{proof}

The \emph{olive property} is a collection of properties of a first-order theory introduced in \cite[Definition~1.8 and Definition~2.1]{Sh:1029}. One talks about the  $(\Delta,\eta,k, m)$-olive property, but for our purposes, $\Delta$ is the set of all formulas in the language of $T$ if $T$ is complete or the set of all quantifier-free formulas otherwise. The parameter $m$ is the arity of the tuples witnessing the order property and can be suppressed. The roles of $\eta$ and $k$ are laid out in \cite[Definition~2.1]{Sh:1029}, and we will say that `$T$ has the olive property' if it has the $(\eta,k)$-olive property for some $\eta\in \{0,1\}^n$ and $k$.   We only remark that by \cite[Theorem~3.1]{Sh:1029} the theory of groups has the olive property (by \cite{Sh:789} the theory of groups fails SOP$_4$), hence (since the olive property for groups is witnessed by quantifier-free formulas), the following applies to it.  

\begin{prop}
		 Suppose that $T$ is a  theory in a countable language with the olive property. If there exists a cardinal $\kappa$ such that $\kappa^+<\fc=\cf(\fc)<2^\kappa$ and $T$ is complete, then $\bbM_T$ does not have a $\preceq$-universal model. 

If the olive property is witnessed by a quantifier-free formulas, then even if $T$ is not complete,  $\bbM_T$ does not even have a $\hookrightarrow$-universal model. 
\end{prop}

\begin{proof} This is \cite[Theorem~2.9 (i) and (ii)]{Sh:1029}. For the quantifier-free case, use part (ii) and the $\lambda-(\eta,k,m)$-olive property from \cite[Definition~2.3 (ii)]{Sh:1029}. 
\end{proof}

With a strengthened cardinal arithmetic assumption one can say more. For a cardinal $\mu$, a set of models  $\cA$ is said to have a \emph{$\preceq$-basis of cardinality~$\mu$}  if there is $\bbB\subseteq \cA$ of cardinality $\mu$ such that  every $A\in \cA$ elementarily embeds into some element of $\bbB$. It is  said to have a \emph{$\hookrightarrow$-basis of cardinality~$\mu$}  if there is $\bbB\subseteq \cA$ of cardinality $\mu$ such that  every $A\in \cA$  embeds (not necessarily elementarily) into some element of $\bbB$.

\begin{prop}
Suppose that $T$ is a theory in a countable language with the SOP, SOP$_4$, or the olive property. If there exists a cardinal $\kappa$ such that $\kappa^+<\fc=\cf(\fc)$ and $\fc^+<2^\kappa$ and $T$ is complete, then $\MODc$ does not have a $\preceq$-basis of cardinality less than $2^\kappa$. 

If the SOP, the SOP$_4$, or the olive property, is witnessed by quantifier-free formulas, then even if $T$ is not complete,   $\MODc(T)$ does not have a $\hookrightarrow$-basis of cardinality less than $2^\kappa$. 
\end{prop}

\begin{proof} 
We will prove the $\preceq$-case, starting with  the following. 

 \begin{claim} \label{C.H}
For every  $\bbA\subseteq \MODc(T)$ with $|\bbA|=2^\kappa$ there are $M_A\in \MODc(T)$ such that $A\prec M_A$ for $A\in \bbA$ and for every choice of  $N_A\in \MODc(T)$ with $M_A\preceq N_A$  for $A\in \bbA$ there exists $X\subseteq \bbA$ of cardinality $2^\kappa$ such that   $N_A$ does not embed into $N_B$ for all distinct  $A$ and $B$ in $X$. 
  \end{claim}
  
\begin{proof} With $\kappa$ as in the assumptions,  a  function $\invp\colon 
\MODc(T)\to [\cP(\kappa)]^{\fc}$  with  the following properties exists. 
\begin{enumerate}
\item If $M_0$ and $M_1$ are in $\MODc(T)$ and $M_0$ is elementarily embeddable  into  $M_1$ then $\invp(M_0)\subseteq \invp(M_1)$. 
\item If the property in question is witnessed by quantifier-free formulas, $M_0$ and $M_1$ are in $\MODc(T)$, and $M_0$ is  embeddable (not necessarily elementarily) into  $M_1$, then $\invp(M_0)\subseteq \invp(M_1)$. 
\item \label{3.H} If $M_0\in \MODc(T)$ and $S\subseteq \kappa$ then there exists $M_1\in \MODc(T)$ such that $M_0\prec M_1$ and $S\in \invp(M_1)$. 
\end{enumerate}
For SOP, $\invp$ is  $\INV(M,\bar C)$ for a fixed $\kappa$-scale $\bar C$ (see \cite[\S 3(4)]{Sh:409} and  \cite[Lemma~3.7]{Sh:409}).  
For SOP$_4$, this is  $\INV_\varphi(M,\bar C)$, where $\varphi$ witnesses SOP$_4$ and~$\bar C$ is a club system (see \cite[Definition~2.13 (b)]{Sh:500}). For the olive property, see~\cite[Remark~1.9]{Sh:1029}.

Fix $\bbA\subseteq \MODc(T)$ with $|\bbA|=2^\kappa$. Let $S_A$, for $A\in \bbA$, be distinct subsets of $\kappa$. By a realizing types argument and \eqref{3.H}, there are $M_A\in \MODc(T)$ such that $A\prec M_A$ and $S_\xi\in \invp(M_A)$. Fix $N_A$ such that $M_A\preceq N_A$. Since $|\inv(N_A)|=\fc$ when $|A|=\fc$, by Hajnal's free subset theorem (\cite[Theorem~1]{hajnal1961proof} applied with $m=2^\kappa$ and $n=\fc$ to the function $A\mapsto \inv(N_A)$)  there exists $X\subseteq \bbA$ of cardinality $2^\kappa$ such that $S_A\notin \inv(N_B)$ for all distinct $A$ and $B$ in $X$, and therefore $N_A$, for $A\in X$, are as required. 
\end{proof}

Suppose towards contradiction that $\MODc$ has a $\preceq$-basis $\bbB$ of cardinality less than $2^\kappa$. Since $\bbM_T$ has $2^\fc$ elements (\cite{Sh:954}), we can fix $\bbA\subseteq \bbM_T$ of cardinality $2^\kappa$.  With $M_A$, for $A\in \bbA$,  as provided by Claim~\ref{C.H}, for every~$A$ there is $N_A\in \bbB$ such that $M_A\preceq N_A$. Since $|\bbB|<2^\kappa$, the conclusion of Claim~\ref{C.H} fails; contradiction.  

A proof of the $\hookrightarrow$-case of Proposition is analogous, using the modification of the Claim in which it is allowed that $N_\xi\hookrightarrow N_\xi'$. 
\end{proof}

\section{Concluding remarks and questions} 
\label{S.Concluding} 

The question that initiated the research reported here remains open

\begin{quest}\label{Q.1} Suppose that there is a  nontrivial countable (or separable) structure $A$ whose theory has the order property and  $\prod_{\cU}A$ is isomorphic to $\prod_{\Fin} A$ for some $\cU\in \beta\bbN\setminus \bbN$. Does it follow that the CH holds? 
\end{quest}

Our main results show that in some models of ZFC in which CH fails the premise of Question~\ref{Q.1} fails as well. The methods of \cite{Sh:509}, \cite{Sh:405}, and \cite{Sh:326} may be  relevant to the  possibility of giving negative answer to  this question. 

 Our proof of Theorem~\ref{2.T.B} uses the well-known technique introduced by Kunen in the proof of  \cite[Theorem~12.7]{kunen1969inaccessibility}.  It may be worth pointing out that, although the proof is well-known, the actual statement of the theorem isn't quite as well-known as it should   be.  This theorem asserts  that  in the standard model for adding $\kappa>\fc$ Cohen reals no well-ordering of $\bbR$ belongs to the $\sigma$-algebra generated by arbitrary rectangles on $\bbR$.  The conclusion of this result is equivalent to the assertion that there are no $\kappa$-chains in  any Borel ordering on a Polish space, and it is often misstated as the weaker assertion  that there are no $\kappa$-chains in $\bbNN/\Fin$.

 The proof of Theorem~\ref{T.A} uses a forcing notion related to the forcing  $\bbC_\kappa$ for adding $\kappa$ side-by-side Cohen reals  and  an  analysis of names which is to some extent similar to Kunen's.   (This forcing  belongs to the class of  \emph{semicohen} forcing notions,  see \cite{Sh:504}.)   The two results are however different, since the forcing~$\bbH_E$ used in the proof of Theorem~\ref{T.A} can add an $\omega_2$-chain to some Borel poset  $(\bbNN,\rho)$  without adding an $\omega_2$-chain  to $(\bbNN,\leq^*)$ (this has been proved for a close relative of $\bbH_E$ in \cite[Theorem~2.1]{Fa:Embedding}).

The argument of the proof of Theorem~\ref{2.T.B} works for many other forcings that add more than $\fc$ reals, as long as one can uniformize the names and there are no ultrafilters on $\bbN$ with small generating sets in the extension. The latter does not apply to the Sacks forcing. As a matter of fact, after adding~$\fc^+$ Sacks reals to a model of CH with countable supports (by either countable support product or countable support iteration), there exists a selective $\aleph_1$-generated ultrafilter on $\bbN$,  and it is a true tie point (\cite{baumgartner1979iterated}). It is therefore not clear whether in some of the Sacks models $\cP(\bbN)/\Fin$ is isomorphic to an ultraproduct of countable atomless Boolean algebras. 
 If so, then this would have to be an $\aleph_1$-generated ultrafilter. The most obvious choice would be an ultrafilter generated by a ground-model selective ultrafilter (there are $2^{\aleph_1}$ such ultrafilters by \cite{baumgartner1979iterated}). As all of these ultrafilters  `look the same' (see \cite{zapletal1999terminal} for an  interpretation of this assertion) this suggests the following test question.   
   \begin{quest} Suppose that in either one of the Sacks models $\cU$ and $\cV$ are $\aleph_1$-generated selective ultrafilters. Is it true that $(\bbN,\leq)^{\cU}\cong (\bbN,\leq)^{\cV}$? 
   \end{quest}
   
   One could ask an analogous question for countable models of other countable first-order theories with the order property; $(\bbN,\leq)$ just appears to provide the simplest interesting instance of this question. The ideas from \cite[\S 2 and \S 4]{Sh:990} may be relevant to this problem in the case of Boolean algebras.

    Question~\ref{Q.Universal} on the existence universal ultrapowers in the absence of CH tackled in \S\ref{S.Universal} also remains open. See \cite{shelah2021divide} for the bigger picture.

We conclude with a few words on `definable' reduced products $\prod_{\cF} A_n$. If~$\cF$ is an analytic filter on $\bbN$ (i.e., one that is analytic as a subset of $\cP(\bbN)$, given its Cantor-set topology) that extends the Fr\'echet filter, then the restriction of $\cF$ to any $\cF$-positive set is not an ultrafilter (because all analytic sets, unlike the nonprincipal ultrafilters,  have the universal  property of Baire.)
Therefore the  Feferman--Vaught theorem (\cite{feferman1959first}, and for the metric case  \cite{ghasemi2016reduced} or \cite[\S 16.3]{Fa:STCstar})  implies that if all $A_n$ are elementarily equivalent, and if $\cF$ is analytic and extends the Fr\'echet filter  then  $\prod_{\cF} A_n$ is elementarily equivalent to $\prod_{\Fin} A_n$. 
Many (but not all) of the reduced products $\prod_{\cF} A_n$ are countably saturated\footnote{A sufficient condition for  countable saturation of $\prod_{\cF} A_n$ was  isolated in \cite[Definition~6.5]{Fa:CH}.} 
 and therefore isomorphic to $\prod_{\Fin} A_n$
if CH holds. The following question is somewhat vague---the answer is clearly a function of the theories of the $A_n$s, among other things--- but see the discussion in the paragraph following it for clarification. 

\begin{quest} \label{Q.filters}Given a sequence $A_n$, for $n\in \bbN$,  of structures of the same language, for what Borel filters $\cF$ on $\bbN$ is  $\prod_{\cF} A_n\cong \prod_{\Fin} A_n$  
provable in ZFC for every sequence of countable (separable) models $A_n$? 
\end{quest}

In the case when each $A_n$ is the two-element Boolean algebra---i.e., the case of the quotients of the form $\cP(\bbN)/\cI$---the isomorphism is provable if and only if there is a continuous $f\colon \cP(\bbN)\to \cP(\bbN)$ that lifts such an isomorphism~(\cite{Sh:987}) and in many (conjecturally, all) cases this is equivalent  to the Rudin--Keisler isomorphism of the underlying ideals (\cite[Corollary~3.4.2]{Fa:AQ} and \cite[Corollary~3]{Sh:987}). For current state of the art in this subject see \cite{farah2022corona}. 

In the case when all $A_n$ are Boolean algebras, Question~\ref{Q.filters} is really a question about abelian \cstar-algebras. This is because the category of Boolean algebras is, via the Stone duality, equivalent to the category of compact, zero-dimensional, Hausdorff spaces and the latter category is, by the Gelfand--Naimark duality, equivalent to the category of unital, abelian, \cstar-algebras (see \cite[\S 1.3]{Fa:STCstar}). 
By this observation and the main result of \cite{FaMcK:Homeomorphisms}, PFA implies that 
two such reduced products are isomorphic if and only if there is an (appropriately defined) `trivial' isomorphism between them. 
For example, PFA implies that $\prod_{\Fin} B\not\cong \cP(\bbN)/\Fin$  if $B$ is the atomless countable Boolean algebra. 
The ultimate extension of the result of \cite{FaMcK:Homeomorphisms} to the coronas of arbitrary  separable \cstar-algebras was proved in~\cite{mckenney2021forcing} and   \cite{vignati2018rigidity}; see also the survey \cite{farah2022corona}. 
   
 \bibliography{all,f1881}
\bibliographystyle{amsplain}

\end{document}